\documentclass[10pt]{amsart}

\usepackage{amssymb,amsmath,amscd,amsthm,amsfonts,wasysym,mathrsfs,amsxtra,stmaryrd}
\input xy
\xyoption{all}

\newcommand{\OO}{\mathscr{O}}

\newcommand{\Cfield}{\mathbb{C}}

\newcommand{\image}{\textnormal{im}\,}
\newcommand{\kernel}{\textnormal{ker}\,}

\newcommand{\Hom}{\textnormal{Hom}}
\newcommand{\dimension}{\textnormal{dim}\,}

\newcommand{\rank}{\textnormal{rk}\,}

\newcommand{\supp}{\textnormal{supp}}

\newcommand{\Ext}{\textnormal{Ext}}
\newcommand{\EExt}{\mathscr{E}xt}

\newcommand{\al}{\alpha}

\newcommand{\Coh}{\mathrm{Coh}}

\newcommand{\arinj}{\ar@{^{(}->}}
\newcommand{\arsurj}{\ar@{->>}}
\newcommand{\areq}{\ar@{=}}

\newcommand{\wh}{\widehat}

\newtheorem{theorem}{Theorem}[section]

\newtheorem{lemma}[theorem]{Lemma}

\newtheorem{proposition}[theorem]{Proposition}
\newtheorem{coro}[theorem]{Corollary}

\newtheorem*{theorem4-13}{Theorem 4.12}
\newtheorem*{theorem4-26}{Theorem 4.26}

\theoremstyle{definition}
\newtheorem{definition}[theorem]{Definition}

\newtheorem{example}[theorem]{Example}

\theoremstyle{remark}
\newtheorem{remark}[theorem]{Remark}

\numberwithin{equation}{section}

\begin{document}

\title{t-structures on elliptic fibrations}

\author[Jason Lo]{Jason Lo}
\address{Department of Mathematics \\
University of Illinois at Urbana-Champaign\\
1409 W Green St\\
Urbana IL 61801 \\
USA}
\email{jccl@illinois.edu}

\keywords{t-structure, torsion pair, elliptic fibration, cohomology, Fourier-Mukai transform}
\subjclass[2010]{Primary 14D06; Secondary: 14F05, 14J27, 14J60}

\begin{abstract}
We consider t-structures that naturally arise on elliptic fibrations.  By filtering the category of coherent sheaves on an elliptic fibration using the torsion pairs corresponding to these t-structures, we prove results describing equivalences of t-structures under Fourier-Mukai transforms.
\end{abstract}

\maketitle
\tableofcontents

\section{Introduction}

This is the first in a series of articles on elliptic fibrations.  In this article, we focus on t-structures that can be constructed using the geometry of an elliptic fibration, and using a relative Fourier-Mukai transform from the derived category of coherent sheaves on the elliptic fibration. In later articles, we will study the relations between these t-structures and those that appear in the study of Bridgeland stability conditions.  We will also study the different notions of stabilities associated with them.

The study of t-structures comes into various problems of active interest in algebraic geometry: for instance, explicitly in the study of Bridgeland stability conditions, and implicitly in the construction of stable sheaves.

In the study of Bridgeland stability conditions (see \cite{StabTC,SCK3,ABL,MacMea,Macri,Schmidt,MacPir1,MacPir2,BMS}), t-structures are part of the definition for a stability condition.  A Bridgeland stability condition on a smooth projective variety $X$ is a pair $(Z,\mathcal A)$ where $\mathcal A$ is the heart of a t-structure on $D(X):=D^b(\Coh (X))$, the bounded derived category of coherent sheaves on $X$, and $Z$ a group homomorphism from the Grothendieck group $K(X)$ to $\Cfield$, with $\mathcal A, Z$ satisfying certain properties.  Therefore, undertanding t-structures on $D(X)$ has implications on the types of stability conditions that can arise, as well as on which objects in $D(X)$ can arise as stable objects.

In the construction of stable sheaves on varieties, t-structures come into play in the method of spectral construction (see \cite{FMW} and, for instance, \cite{FMNT,CDFMR,HRP,ARG}), albeit implicitly.  For instance, when $X$ admits the structure of an elliptic fibration $\pi : X \to S$, if we have a `dual' fibration $\wh{\pi} : Y \to S$ where $Y$ is another elliptic fibration with a Fourier-Mukai transform $\Phi : D(Y) \to D(X)$, then the spectral construction produces stable sheaves on $X$ of the form $\Phi (F)$ where $F$ is a coherent sheaf on $Y$ supported in codimension 1.  In our notation in Section \ref{section-induced} below, we can view the sheaf $\Phi (F)$ as lying in the category $\Phi (\{\Coh^{\leq 0}(Y_s)\}^\uparrow)$, where $\{ \Coh^{\leq 0}(Y_s)\}^\uparrow$ is contained in the heart of a t-structure that is a tilt of the standard t-structure on $D(Y)$.

In this paper, we consider various t-structures on $D(X)$ where $X$ is a smooth projective variety that comes with an elliptic fibration $\pi : X \to S$ and a Fourier-Mukai partner $\wh{\pi} : Y \to S$ (see Section \ref{beginning} for the precise assumptions).  These t-structures arise from the geometry of $X$ itself, the geometry of the fibration $\pi$, and also from the Fourier-Mukai transform between $D(X)$ and $D(Y)$.  Each of these t-structures  corresponds to a torsion pair in $\Coh (X)$ or $\Coh (Y)$, which in turn is determined by the torsion class in the torsion pair.

When $X,Y$ are elliptic threefolds, the torsion classes in $\Coh (X)$ from which we build all the other torsion classes  in this article (by taking intersections and extensions) can be summarised in the following diagram, where each arrow denotes an inclusion of categories:
\[
\xymatrix{
 & & & \mathfrak T_X \\
 W_{0,X} \ar[rrr] & & & W_{0,X}' \ar[u] \\
 \{\Coh^{\leq 0}(X_s)\}^\uparrow \ar[u] \ar[rrr] & & & \Coh^{\leq 2}(X) \ar[u] \\
 & & \Coh^{\leq 1}(X) \ar[ur] \ar[r] & \Coh (\pi)_{\leq 1} \ar[u] \\
 & & & \\
 \Coh^{\leq 0}(X) \ar[uuu] \ar[r] \ar[uurr] & \mathcal B_{X,\ast} \ar[r] & \Coh(\pi)_{\leq 0} \ar[uu] &
}.
\]
Here,
\begin{itemize}
\item $\Coh^{\leq d}(X)$ is the category of coherent sheaves $E$ on $X$ supported in dimension at most $d$;
\item $\Coh (\pi)_{\leq d}$ is the category of coherent sheaves $E$ on $X$ such that the dimension of $\pi (\text{supp}(E))$ is at most $d$;
\item $\{\Coh^{\leq 0}(X)\}^\uparrow$ is the category of coherent sheaves $E$ on $X$ such that the restriction $E|_s$ to the fiber over any closed point $s \in S$ is supported in dimension 0;
\item $W_{0,X}$ is the category of coherent sheaves $E$ on $X$ such that $\Psi (E)$ is isomorphic to a coherent sheaf on $Y$ sitting at degree 0 (i.e.\ the category of the `$\Psi$-WIT$_0$ sheaves'), with $\Psi$ being the Fourier-Mukai trasform $D(X) \to D(Y)$ coming from the construction of $Y$;
\item $W_{0,X}'$ is the category of coherent sheaves $E$ on $X$ such that, for a general closed point $s \in S$, the restriction $E|_s$ to the fiber over $s$ is taken by $\Psi_s$ (the base change of $\Psi$ under the closed immersion $\{s\}\hookrightarrow S$) to a sheaf sitting at degree 0;
\item $\mathfrak{T}_X$ is the extension closure of $W_{0,X}'$ and $\Phi (\{\Coh^{\leq 0}(Y_s)\}^\uparrow)$.
\item $\mathcal B_{X,\ast}$ is the extension closure of $\Coh^{\leq 0}(X)$ and a category of fiber sheaves - see \eqref{eq64}.
\end{itemize}

Using these subcategories, we filter the category of coherent sheaves on $X$ or $Y$ into smaller pieces with distinct geometric properties, the behaviors of which under the Fourier-Mukai transform $\Psi : D(X) \to D(Y)$ are tractable.  As a consequence, we obtain results that explicitly describe how  t-structures on $D(X)$ and $D(Y)$ are equivalent, or related by tilts, under a relative  Fourier-Mukai transform $\Psi : D(X) \to D(Y)$.

Another underlying motivation for studying the t-structures that appear in this paper is, to some extent, to mitigate the problem that stability is not always well-behaved under base change.  Given a flat morphism $\pi : X \to S$ of smooth projective varieties and a torsion-free coherent sheaf $E$ on $X$, there are various results on how the  stability (in the sense of slope stability or Gieseker stability) of $E$ is related to the  stability  of the restriction of $E$ to the generic fiber of $\pi$ - see \cite[Proposition 7.1]{FMTes}, \cite[Lemma 2.1]{BMef}, and \cite[Proposition 3.3]{ARG}.  In general, however, there seems to be few results relating  the stability of $E$ and the stability of its restriction to a special fiber.

To this end, we propose to replace the notion of stability for coherent sheaves on $X$ by the notion of being $\Psi$-WIT$_0$.   As observed in Lemma \ref{lemma20}, being WIT$_0$ is compatible with base change when $\pi$ is of relative dimension 1, and so, compared to slope stability, it gives us a better behaved notion under the Fourer-Mukai transform $\Psi$.  More precisely, when $\pi$ is an elliptic fibration and $F$ is a sheaf supported on a fiber of $\pi$, that $F$ is WIT$_0$ is equivalent to all its Harder-Narasimhan (HN) factors having strictly positive slopes  \cite[Corollary 3.29, Proposition 6.38]{FMNT}.  Therefore, when a sheaf $E$ on $X$ is $\Psi$-WIT$_0$, we can restrict it to a special fiber and still understand the HN filtration after restriction.  As a result, even though special cases of the torsion classes $\Coh^{\leq d}(X)$ and $\Coh (\pi)_{\leq d}$ have appeared in the author's earlier article \cite{Lo5} (as the torsion classes $\mathcal T_X$ and $\mathcal B_X$), by taking the above perspective in this article, we obtain finer results on the behavior of these torsion classes under Fourier-Mukai transforms.

\subsection{Main Results}

Let us write $\mathfrak{C}_X$ to denote the heart of the t-structure on $D(X)$ given by
\begin{align*}
  D^{\leq 0}_{\mathfrak C} &:= \{ E \in D(X) : H^0(E) \in \mathfrak T_X, H^i(E) =0 \, \forall \, i >0 \}, \\
  D^{\geq 0}_{\mathfrak C} &:= \{ E \in D(X) : H^{-1}(E) \in \mathfrak T_X^\circ, H^i(E)=0 \, \forall\, i < -1\},
\end{align*}
where $\mathfrak T_X^\circ$ is the torsion-free class in $\Coh (X)$ corresponding to the torsion class $\mathfrak T_X$, and write $\mathfrak D_Y$ to denote the heart of the t-structure on $D(Y)$ given by
\begin{align*}
  D^{\leq 0}_{\mathfrak D} &:= \{ E \in D(Y) : H^0(E) \in \mathcal B_{Y,\ast}, H^i(E) =0 \, \forall \, i>0 \}, \\
  D^{\geq 0}_{\mathfrak D} &:= \{ E \in D(Y) : H^{-1}(E) \in \mathcal B_{Y,\ast}^\circ, H^i(E)=0\, \forall \, i<-1\},
\end{align*}
where $\mathcal B_{Y,\ast}^\circ$ is the torsion-free class in $\Coh (Y)$ corresponding to the torsion class $\mathcal B_{Y,\ast}$.  That is,
\begin{equation*}
\mathfrak{C}_X = \langle \mathfrak T_X^\circ[1], \mathfrak T_X \rangle \text{\quad and \quad} \mathfrak D_Y = \langle \mathcal B_{Y,\ast}^\circ [1], \mathcal B_{Y,\ast}\rangle.
\end{equation*}
Let us also write  $\Lambda$ to denote the composition of the Fourier-Mukai functor $\Psi (-)$ with the derived dual functor $-^\vee$, i.e.\ $\Lambda (-) = (\Psi (-))^\vee$.  Then we have:

\begin{theorem4-13}
When $X$ is a smooth elliptic surface, the functor $\Lambda$ induces an equivalence between the t-structure $(D^{\leq 0}_{\mathfrak C}, D^{\geq 0}_{\mathfrak C})$ on $D(X)$, and the t-structure $(D^{\leq 0}_{\mathfrak D}, D^{\geq 0}_{\mathfrak D})$ on $D(Y)$.  Equivalently, $\Lambda$ induces an equivalence of hearts
\begin{equation}
   \mathfrak C_X \overset{\thicksim}{\to} \mathfrak D_Y [-1].
\end{equation}
\end{theorem4-13}

Theorem \ref{theorem1} for elliptic surfaces can be considered as a special case of a  result of Yoshioka's \cite[Proposition 3.3.5]{YosPII} - see the end of Section \ref{sec-ellipsurf}, including Lemma \ref{lemma76}, for a precise explanation of this.  Yoshioka's result is more general in the case of elliptic surfaces, as it is stated for tilts of categories of perverse sheaves, as opposed to tilts of categories of coherent sheaves.  The result  \cite[Proposition 3.3.5]{YosPII} was obtained by Yoshioka as part of an argument towards showing an isomorphism between moduli spaces of twisted stable sheaves on elliptic fibrations that are Fourier-Mukai partners (see \cite[Proposition 3.4.4]{YosPII} or \cite[Theorem 3.15]{YosAS}).

When $X$ and $Y$ are elliptic threefolds, the functor $\Lambda$ no longer induces an equivalence between the t-structures $(D^{\leq 0}_{\mathfrak C}, D^{\geq 0}_{\mathfrak C})$ and $(D^{\leq 0}_{\mathfrak D}, D^{\geq 0}_{\mathfrak D})$.  Instead,  the hearts of these two t-structures differ by a tilt (up to a shift):

\begin{theorem4-26}
Suppose $X$ is a smooth elliptic threefold.  Then the heart $\Lambda (\mathfrak C_X)$ differs from  the heart $\mathfrak D_Y[-2]$  by one tilt.
\end{theorem4-26}

We prove Theorem \ref{theorem2} by explicitly studying the effects of $\Lambda$ on various subcategories of $\Coh (X)$, which arise from a nested sequence of Serre subcategories of $\Coh (X)$ - see Theorem \ref{theorem3} and its proof.

The implications of our results on moduli spaces will be addressed in a later paper.  The relationships between the t-structures considered in this paper, as well as other t-structures that come up in the study of Bridgeland stability conditions, will also be explored in a later article.

Some of the t-structures considered in this article were already considered in \cite{Lo9} and \cite{Lo10}, which grew out of an attempt to understand the results in \cite{FMTes} and \cite{BMef}.

\noindent
\textbf{Acknowledgements.} The author would like to thank the National Center for Theoretical Sciences in Taipei, as well as the Max Planck Institute for Mathematics in Bonn, for their hospitality during the author's stay from March through May 2014, when most of this project took place.  He would  like to thank Ziyu Zhang for many invaluable comments on various versions of this article, and Tom Nevins and Sheldon Katz for comments on a later version.  He would also like to thank the referee for a careful reading of the manuscript, and many helpful comments leading to improvement on the exposition in this article.

\section{Notation}\label{beginning}

The setup considered in the rest of this article (except for Section \ref{sec-bcf}, where the setup is slightly more general) is as follows: we will assume that we have a pair of morphisms of smooth projective varieties $\pi : X \to S$ and $\hat{\pi} : Y \to S$ that satisfies the following conditions:
\begin{itemize}
\item[(i)] There is a pair of relative integral functors that are quasi-inverse to each other, up to a shift (so that they are necessarily equivalences):
\begin{equation*}
  \Psi : D^b(X) \overset{\thicksim}{\to} D^b(Y) \text{\qquad and \qquad} \Phi : D^b(Y) \overset{\thicksim}{\to} D^b(X).
\end{equation*}
\item[(ii)] The morphisms $\pi, \hat{\pi}$ are both flat.
\end{itemize}
Note that, by property (i) and our assumption that $X, Y, S$ are all projective, the kernels of the relative integral functors $\Psi$ and $\Phi$ both have finite homological dimensions, as complexes of $\OO_X$-modules or $\OO_Y$-modules, respectively \cite[Proposition 2.10]{RMS}.  This ensures,  that given any morphism of varieties $S' \to S$, the corresponding base changes $\Psi_{S'} : D(X_{S'}) \to D(Y_{S'})$ and $\Phi_{S'} : D(Y_{S'}) \to D(X_{S'})$
   still take the bounded derived categories $D^b(X_{S'})$ and $D^b(Y_{S'})$ into each other \cite[Section 6.1.1]{FMNT}.

Let  $\pi, \wh{\pi}$ be as above.  The following notations will be used throughout this article:
\begin{enumerate}
\item For any variety $W$, we will write $D(W)=D^b(\Coh (W))$ to denote the bounded derived category of coherent sheaves on $W$ unless otherwise stated.
\item For any closed point $s \in S$, we will write $\iota_s$ (resp.\  $j_s$) to denote the closed immersion of the fiber $X_s \to X$ (resp.\ $Y_s \to Y$) of $\pi$ (resp.\ $\wh{\pi}$) over  $s$.  When $E$ is a coherent sheaf on $X$ (resp.\ on $Y$), we will write $E|_s$ to denote the restriction $\iota_s^\ast E$ (resp.\ the restriction $j_s^\ast E$), and write $E|_s^L$ to denote the derived restriction $L\iota_s^\ast E$ (resp.\ the derived restriction $Lj_s^\ast E$).
\item We will write $\mathcal B_X$ to denote the category of coherent sheaves $E$ on $X$ such that $E|_s$ is zero for a general closed point $s \in S$.  We similarly define $\mathcal B_Y$.
\item For any Abelian category $\mathcal A$ and any $E \in D(\mathcal A)$, we will write $H^i (E)$ to denote the degree-$i$ cohomology of $E$ with respect to the standard t-structure on $D(\mathcal A)$.  When $\mathcal B$ is the heart of a t-structure on $D(\mathcal A)$,  for  any $E \in D(\mathcal A)$, we will write $H^i_{\mathcal B}(E)$ to denote the degree-$i$ cohomology of $E$ with respect to the t-structure with heart $\mathcal B$.  We will also define, for any integers $j,k$,
\[
  D^{[j,k]}_{\mathcal B} := D^{[j,k]}_{\mathcal B}(\mathcal A) := \{ E \in D(\mathcal A) : H^i_{\mathcal B}(E)=0 \text{ for all }i \notin [j,k]\}.
\]
\item For each integer $i$, we will write $W_{i,X}$ to denote the category of coherent sheaves $E$ on $X$ such that
\begin{equation}\label{eq62}
\Psi (E) \cong \wh{E} [-i]
 \end{equation}
 for some $\wh{E} \in \Coh (Y)$, and refer to objects in $W_{i,X}$ as $\Psi$-WIT$_i$ sheaves on $X$.  For a $\Psi$-WIT$_i$ sheaf $E$ on $X$, we will refer to $\wh{E}$ in \eqref{eq62} as the transform of $E$.  We similarly define $W_{i,Y}$ and $\Phi$-WIT$_i$ sheaves for any integer $i$, with $\Psi$ replaced by $\Phi$ and $X$ replaced by $Y$ in the definitions above.
\item (WIT$_i$ sheaves) For any integer $i$, we will write $\Psi^i (-)$ to denote the composite functor $H^i (\Psi (-))$, where $H^i$ is the cohomology functor with respect to the standard t-structure on $D(Y)$.  Similarly, we write $\Phi^i (-) := H^i (\Phi (-))$.
\item If a coherent sheaf $E$ on $X$ is supported on a finite number of fibers of $\pi$, then we will refer to it as a fiber sheaf.  Similarly for coherent sheaves on $Y$.
\item (Torsion pairs) Given an Abelian category $\mathcal A$, recall that a torsion pair $(\mathcal T, \mathcal F)$ in $\mathcal A$ is a pair of full subcategories of $\mathcal A$ satisfying the following two conditions:
    \begin{enumerate}
    \item Every $E \in \mathcal A$ fits in a short exact sequence in $\mathcal A$
        \[
        0 \to T \to E \to F \to 0
        \]
        where $T \in \mathcal T$ and $F \in \mathcal F$;
    \item $\Hom_{\mathcal A}(T,F)=0$ for any $T \in \mathcal T, F \in \mathcal F$.
    \end{enumerate}
    Whenever we have a torsion pair $(\mathcal T, \mathcal F)$ in an Abelian category $\mathcal A$, we will refer to $\mathcal T$ as the torsion class of the torsion pair, and $\mathcal F$ as the torsion-free class of the torsion pair.  We will say that a subcategory $\mathcal C$ of an Abelian category $\mathcal A$ is a torsion class if it is the torsion class of a torsion pair in $\mathcal A$.
\item Whenever we have a proper morphism $f : V \to W$ of noetherian schemes, we will write $\Coh(f)_b$ to denote the category of coherent sheaves $E$ on $V$ such that $\dimension (f (\text{supp}(E)))=b$ for any $b\geq 0$.  We will also write  $\Coh (f)_{\leq b}$ to denote the category of coherent sheaves $E$ on $X$ with $\dimension (f (\supp (E))) \leq b$.  For any $b \geq 0$, the category $\Coh (f)_{\leq b}$ is a Serre subcategory (i.e.\ closed under subobjects, extensions and quotients in $\Coh(V)$), and in particular,  is the torsion class of a torsion pair in $\Coh (V)$.  The category $\Coh (f)_0 = \Coh (f)_{\leq 0}$ is precisely the category of fiber sheaves on $V$.  Also, when $f$ is flat of relative dimension 1 and $W$ is irreducible of dimension $d$, we have $\Coh (f)_{\leq d} = \mathcal B_V$.
\item For any variety $W$, when $\mathcal C_1,\cdots, \mathcal C_n$ are subcategories of  $D(W)$, we will write $\langle \mathcal C_1, \cdots, \mathcal C_n \rangle$ to denote the extension closure generated by the $\mathcal C_i$ in $D(W)$.
\item For any noetherian scheme $W$ and any integer $d \geq 0$, we will write $\Coh^{\leq d}(W)$ to denote the category of coherent sheaves on $W$ supported in dimension at most $d$, write $\Coh^{=d}(W)$ to denote the category of pure $d$-dimensional coherent sheaves, and write $\Coh^{\geq d}(W)$ to denote the category of coherent sheaves with no nonzero subsheaves supported in dimension $d-1$ or lower.
\item When $W$ is a smooth projective variety, we sometimes write $\mathcal T_W$ to denote the category of coherent sheaves on $W$ that are torsion, and write $\mathcal F_W$ to denote the category of coherent sheaves on $W$ that are torsion-free.
\item For a fixed variety $W$ and a full subcategory $\mathcal C$ of $\Coh (W)$, we will define
    \[
      \mathcal C^\circ := \{ E \in \Coh (W) : \Hom_{\Coh (W)}(C,E)=0 \text{ for all } C \in \mathcal C\}.
    \]
\item In a noetherian Abelian category $\mathcal A$ (such as $\Coh (W)$ for a noetherian scheme $W$), whenever we have a full subcategory $\mathcal C \subseteq \mathcal A$ that is closed under extensions and quotients in $\mathcal A$, we have a torsion pair $(\mathcal C, \mathcal C^\circ)$ in $\mathcal A$ by \cite[Lemma 1.1.3]{Pol}.
\item Whenever we have a torsion pair $(\mathcal T, \mathcal F)$ in an Abelian category $\mathcal A$, there is a corresponding t-structure $(D^{\leq 0}, D^{\geq 0})$ on the derived category $D(\mathcal A)$ given by
    \begin{align*}
      D^{\leq 0} &:= \{ E \in D(\mathcal A) : H^0(E) \in \mathcal T, H^i(E)=0 \text{ for all $i>0$}\},\\
      D^{\geq 0} &:= \{ E \in D(\mathcal A) : H^{-1}(E) \in \mathcal F, H^i(E)=0 \text{ for all $i<-1$ }\}.
    \end{align*}
\item (Elliptic fibrations) By an elliptic fibration, we will mean a proper flat morphism $\pi : X \to S$ such that the generic fiber of $\pi$ is a smooth elliptic curve.  (By the flatness of $\pi$, it follows that all fibers of $\pi$ are 1-dimensional.)  We will refer to $\pi$ (or $X$) as an elliptic surface when $\dimension X=2$, and as an elliptic threefold when $\dimension X=3$.

    In addition, we will say $\wh{\pi}$ is a dual elliptic fibration, or say $\pi$ and $\wh{\pi}$ are a pair of dual elliptic fibrations, when $\pi, \wh{\pi}$ are elliptic fibrations of the same dimension satisfying conditions (i) and (ii), where the kernels of $\Psi$ and $\Phi$ in condition (i) are both coherent sheaves sitting at degree 0, flat over both $X$ and $Y$, and we have $\Phi \Psi = \mathrm{id}_{D(X)} [-1], \Psi \Phi = \mathrm{id}_{D(Y)}[-1]$.

    Note that, since $\Psi$ and $\Phi$ are assumed to be relative integral functors, all 0-dimensional sheaves on $X$ are  $\Psi$-WIT$_0$, and are taken to pure 1-dimensional fiber sheaves on $Y$ which are $\Phi$-WIT$_1$.
\end{enumerate}

\begin{example}
The prototypical examples of dual elliptic fibrations $\pi : X \to S$ and $\wh{\pi}: Y \to S$ satisfying our definition above include:
 \begin{itemize}
 \item Elliptic surfaces $\pi : X \to S$ considered by Bridgeland in \cite{FMTes}, or elliptic threefolds $\pi : X \to S$ considered by Bridgeland-Maciocia in \cite{BMef}. In both cases, the fibration $\wh{\pi} : Y \to S$ is constructed as a relative moduli space of stable sheaves on the fibers of $\pi$, and the singular fibers of $\pi$ are not necessarily irreducible.  If $\mathscr P$ denotes the universal sheaf on $Y \times X$ for the above moduli problem, then the relative integral functor $\Psi : D(X) \to D(Y)$ with kernel $\mathscr P$ is a Fourier-Mukai transform.
 \item Weierstrass  fibrations $\pi : X \to S$ (which are elliptic fibrations) in the sense of \cite[Section 6.2]{FMNT}, where all the fibers are Gorenstein and geometrically integral. In this case, $Y$ can be taken as the Altman-Kleiman compactification of the relative Jacobian of $X$, and $\Psi : D(X) \to D(Y)$ taken to be the relative Fourier-Mukai transform with the relative Poincar\'{e} sheaf as the kernel.
 \end{itemize}
 In both cases, a quasi-inverse $\Phi : D(Y) \to D(X)$ can always be constructed making $\pi, \wh{\pi}$ a pair of dual fibrations in the sense  above.  In particular, the kernels of $\Psi$ and $\Phi$ are both coherent sheaves sitting at degree 0, flat over both factors of $X \times Y$.
\end{example}

\section{General constructions on fibrations}

\subsection{Base change formulas}\label{sec-bcf}

Suppose $\pi : X \to S$ and $\wh{\pi} : Y \to S$ are a pair of proper morphisms of varieties  satisfying properties (i) and (ii) as in the beginning of Section \ref{beginning}.  Then we have the following base change formula \cite[(6.3)]{FMNT}:
\begin{equation}\label{eq6}
{j_s}_\ast \Psi_s (E) \cong \Psi ({\iota_s}_\ast E) \text{\quad for all $E \in D(X_s)$}.
\end{equation}
  Note that, this is where condition (ii) comes in, since \eqref{eq6} depends on the morphism $\hat{\pi}$ being flat.

  Assuming additionally that  the kernel for the relative integral functor $\Psi$ (resp.\ $\Phi$) has finite Tor-dimension as a complex of $\OO_X$-modules (resp.\ $\OO_Y$-modules), we have the following well-known observation as a consequence of the base change \eqref{eq6}; we omit its proof:

\begin{lemma}\label{lemma14}
For every closed point $s \in S$, the induced integral functors $\Psi_s : D(X_s) \to D(Y_s)$ and $\Phi_s : D(Y_s) \to D(X_s)$ are equivalences.
\end{lemma}

The following is a second  base change formula useful to us, which depends on $\pi$ being flat \cite[(6.2)]{FMNT}: with $\iota_s, j_s$ as above, for any $E \in D(X)$ we have
\begin{equation}\label{eq57}
Lj_s^\ast \Psi (E) \cong \Psi_s (L\iota_s^\ast E)
\end{equation}
This leads to the following observation that we will use frequently:

\begin{lemma}\label{lemma66}
For any $E \in D(X)$, we have $\pi (\supp (E)) = \hat{\pi} (\supp (\Psi E))$.
\end{lemma}

\begin{proof}
Take any $s \in S \setminus \pi (\supp (E))$.  Then $0=E|^L_s$, and we have $0 =\Psi_s (E|^L_s) \cong (\Psi E)|^L_s$ by the base change \eqref{eq57}, i.e.\ $s \in S \setminus \hat{\pi} (\supp (\Psi E))$.  In other words, we have $\hat{\pi} (\supp (\Psi E)) \subseteq \pi (\supp (E))$.  By symmetry, we have equality.
\end{proof}

An immediate consequence of the base change formula \eqref{eq6} is the following:

\begin{lemma}\label{lemma15}
If $l,m$ are  integers such that $\Psi^i (E)=0$ for all $i \notin [l,m]$ and for all $E \in \Coh (X)$, then we have $\Psi_s^i (F)=0$ for any closed point $s \in S$, $i \notin [l,m]$ and any $F \in \Coh (X_s)$.
\end{lemma}
As a result, if $\Psi$ is a relative integral functor that takes coherent sheaves on $X$ to $n$-term complexes, then for any closed point $s \in S$, the base change $\Psi_s$  also takes coherent sheaves on the fiber $X_s$ to $n$-term complexes.

\begin{remark}\label{remark3}
From \cite[Corollary 6.3]{FMNT}, we know that, if $n := p + m_0$ where $p$ is the dimension of the fibers of $\pi$, and $m_0$ is the largest index $m$ such that $H^m(K)=0$, where $K \in D(X \times_S Y)$ is the kernel of the relative integral functor $\Psi$, then we have, for any $E \in \Coh (X)$, the base change
\begin{equation}\label{eq7}
 \Psi^n (E)|_s \cong \Psi^n_s (E|_s)  \text{\quad for any closed point $s \in S$}.
\end{equation}
Given Lemma \ref{lemma15}, we can think of the base change \eqref{eq7} as saying: for a coherent sheaf $E$ on $X$, `the right-most cohomology of $\Psi (E)$ vanishes if and only if the same holds on each fiber'.
\end{remark}

Borrowing notation from \cite{Lo9}, we  define, for any base change morphism $S' \to S$,  the subcategory of $\Coh (X_{S'})$
\[
  \mathcal B_{i,X_{S'}} := \{ E \in \Coh (X_{S'}) : \Psi_{S'}^i (E)=0 \}
\]
for any integer $i$.  We similarly define $\mathcal B_{i,Y_{S'}}$ (using the vanishing of $\Phi_{S'}^i$) for any morphism $S' \to S$. The interpretation of \eqref{eq7} at the end of Remark \ref{remark3} can now be stated precisely as follows:

\begin{lemma}\label{lemma19}
Let $n$ be as in Remark \ref{remark3}.  Then for any $E \in \Coh (X)$, we have $E \in \mathcal B_{n,X}$ if and only if $E|_s \in \mathcal B_{n, X_s}$ for every closed point $s \in S$.
\end{lemma}

\begin{proof}
From Lemma \ref{lemma14}, we know $\Psi_s$ is an equivalence.  The claim then follows from \eqref{eq7}.
\end{proof}

When the morphism $\pi$ has relative dimension 1 and the kernel of $\Psi$ is a sheaf sitting at degree 0, we have $n=1$ where $n$ is as in Remark \ref{remark3}.  It then follows that $\mathcal B_{1,X}=W_{0,X}$ (and similarly for $Y$).  In other words, we have the following interpretation of WIT$_0$ sheaves on fibrations of relative dimension 1:

\begin{lemma}\label{lemma20}
Suppose $\pi$ has relative dimension 1, and the kernel of the integral functor $\Psi$ is a sheaf (sitting at degree 0).  Then for any $E \in \Coh (X)$, we have that $E$ is $\Psi$-WIT$_0$ if and only if $E|_s$ is $\Psi_s$-WIT$_0$ for every closed point $s \in S$.
\end{lemma}

\begin{proof}
This follows from Lemma \ref{lemma15}, together with Lemma \ref{lemma19} with $n=1$.
\end{proof}

\begin{remark}
Though innocuous-looking, Lemma \ref{lemma20} is a key lemma in this article.  On an elliptic fibration $\pi : X \to S$, the stability of a sheaf $F$ on $X$ is related to the stability of the restriction of $F$ to the generic fiber of $\pi$ - but this relation often depends on the Chern classes of $F$ (see \cite[Section 7.1]{FMTes} or \cite[Lemma 2.1]{BMef}).  By replacing stability with WIT$_i$ properties, we obtain a framework that is more compatible with base change.  For a fiber sheaf on an elliptic fibration that possesses a dual fibration, being WIT$_i$is inherently related to the structure of its HN filtration with respect to slope stability on the fibers (see \cite[Corollary 3.29]{FMNT}).
\end{remark}

\subsection{Torsion pairs induced from fibers}\label{section-induced}

Given any morphism of algebraic varieties $\pi : X \to S$, we describe here two recipes for constructing torsion pairs: one restricts a torsion pair on $X$ to torsion pairs on the fibers of $\pi$, while the other gives a torsion pair on $X$ induced from torsion pairs on the fibers of $\pi$.

Take any subcategory $\mathcal T$ of $\Coh (X)$, and fix any closed point $s \in S$.  Let $\iota_s$ denote the inclusion of the fiber $X_s \hookrightarrow X$.  Consider the following two subcategories of $\Coh (X_s)$:
\begin{align*}
  \mathcal T|_s &:= \{ F \in \Coh (X_s) : F \cong E|_s \text{ for some } E \in \mathcal T\}; \\
  \mathcal T' &:= \{ F \in \Coh (X_s) : \text{ there exists } E \twoheadrightarrow {\iota_s}_\ast F \text{ in $\Coh (X)$ for some $E \in \mathcal T$}\}.
\end{align*}
The inclusion $\mathcal T|_s \subseteq \mathcal T'$ is clear.  When $\mathcal T$ is closed under taking quotients in $\Coh (X)$, we also have the inclusion $\mathcal T' \subseteq \mathcal T|_s$: for any $F \in \mathcal T'$, suppose $E$ is an object in $\Coh (X)$ such that we have a surjection $E \twoheadrightarrow {\iota_s}_\ast F$ in $\Coh (X)$.  Then ${\iota_s}_\ast F$ lies in $\mathcal T$, and we have  $F \cong ({\iota_s}_\ast F)|_s$.

In particular,  when $\mathcal T$ is the torsion class of a torsion pair in $\Coh (X)$, the two subcategories $\mathcal T|_s$ and $\mathcal T'$ coincide.

\begin{lemma}\label{lemma16}
 Let $\mathcal T$ be a torsion class  in $\Coh (X)$.  Then, for each closed point $s \in S$, the category $\mathcal T|_s$ is a torsion class  in $\Coh (X_s)$.
\end{lemma}

\begin{proof}
To show that $\mathcal T|_s$ is a torsion class, it suffices to check that it is closed under quotients and extensions in $\Coh (X_s)$.  That $\mathcal T|_s$ is closed under quotients is clear from the description $\mathcal T|_s = \mathcal T'$ above.  Now, suppose we have $F_1, F_2 \in \mathcal T|_s$ and $F_i \cong E_i|_s$ for some $E_i \in \mathcal T$ for $i=1,2$.  Consider any extension in $\Coh (X_s)$
\[
0 \to F_1 \to F \to F_2 \to 0,
\]
which pushforwards to a short exact sequence in $\Coh (X)$
\[
0 \to {\iota_s}_\ast F_1 \to {\iota_s}_\ast F \to {\iota_s}_\ast F_2 \to 0.
\]
For each $i$, we have  ${\iota_s}_\ast F_i \in \mathcal T$, and so ${\iota_s}_\ast F$ also lies in $\mathcal T$.  Then $F \in \mathcal T|_s$ since $F \cong ({\iota_s}_\ast F)|_s$.
\end{proof}

\begin{definition}\label{def1}
Suppose that, for each closed point $s \in S$, we are given a subcategory $\mathcal T_s$ of $\Coh (X_s)$.  Then we set
\begin{equation}
  \{\mathcal T_s\}^\uparrow := \{ E \in \Coh (X) : E|_s \in \mathcal T_s \text{ for all closed points $s \in S$}\}.
\end{equation}
\end{definition}


\begin{lemma}\label{lemma17}
Suppose that, for each closed point $s \in S$, we have a torsion class $\mathcal T_s$ in $\Coh(X_s)$.  Then the category $\{\mathcal T_s\}^\uparrow$ is the torsion class of a torsion pair in $\Coh (X)$.
\end{lemma}

\begin{proof}
Let us write $\mathcal T$ to denote $\{\mathcal T_s\}^\uparrow$.  It suffices to check that $\mathcal T$ is closed under quotients and extensions in $\Coh (X)$.  That $\mathcal T$ is closed under quotients is clear.  Now, take any $E_1, E_2 \in \mathcal T$ and consider the extension
\[
  0 \to E_1 \to E \to E_2 \to 0 \text{\quad in $\Coh (X)$}.
\]
Fixing any closed point $s \in S$ and restricting the above short exact sequence to $X_s$, we get the exact sequence
\[
  E_1|_s \overset{\al}{\to} E|_s \to E_2|_s \to 0 \text{\quad in $\Coh (X_s)$}.
\]
Since $\mathcal T_s$ is closed under quotients, the image of $\al$ lies in $\mathcal T_s$.  Then, because $\mathcal T_s$ is closed under extensions, we have  $E|_s \in \mathcal T_s$.
\end{proof}

\begin{remark}\label{remark4}
Given the constructions in Lemmas \ref{lemma16} and \ref{lemma17}, it is natural to ask: are the two constructions described mutually inverse?  In other words:
\begin{itemize}
\item[(a)] Given a torsion class $\mathcal T$ on $X$, do we have $\{ \mathcal T|_s \}^\uparrow = \mathcal T$?
\item[(b)] Given a torsion class  $\mathcal T_s$ in $\Coh (X_s)$ for each closed point $s \in S$,  do we have $(\{ \mathcal T_s\}^\uparrow)|_s = \mathcal T_s$ for each closed point $s \in S$?
\end{itemize}
In Lemma \ref{lemma18} below, we show that the answer to the question (b) is `yes'.  On the other hand, even though we always have the inclusion $\mathcal T \subseteq \{\mathcal T|_s\}^\uparrow$, without further assumptions on the varieties $X, S$ or the morphism $\pi$, the answer to the question (a) is \textit{a priori} a `no'.
\end{remark}

\begin{lemma}\label{lemma18}
Suppose that, for each closed point $s \in S$, we have a torsion class $\mathcal T_s$ in $\Coh(X_s)$.  Then
\begin{equation*}
   (\{\mathcal T_s\}^\uparrow)|_s = \mathcal T_s
\end{equation*}
for each closed point $s \in S$.
\end{lemma}

\begin{proof}
For any closed point $s \in S$ and $F \in \mathcal T_s$, we have ${\iota_s}_\ast F \in \{\mathcal T_s\}^\uparrow$.  Hence $F \cong \iota_s^\ast({\iota_s}_\ast F)$ lies in $(\{\mathcal T_s\}^\uparrow)|_s$.  The other inclusion follows directly from the definitions.
\end{proof}

\begin{example}\label{example1}
Let $\pi : X \to S$ and $\hat{\pi} : Y \to S$ be a pair of proper morphisms satisfying conditions (i) and (ii)  laid out in the beginning of Section \ref{beginning}, and let $n$ be as in Remark \ref{remark3}.  Then we have:
\begin{itemize}
\item[(a)] $\mathcal B_{n,X}|_s = \mathcal B_{n,X_s}$ for any closed point $s \in S$, and
\item[(b)] $\{ \mathcal B_{n,X} |_s \}^\uparrow=\{\mathcal B_{n,X_s}\}^\uparrow  = \mathcal B_{n,X}$.
\end{itemize}
To see why (a) holds, fix any closed point $s \in S$.  That $\mathcal B_{n,X}|_s \subseteq \mathcal B_{n,X_s}$ follows from \eqref{eq7}.  To show the other inclusion, take any $F \in \mathcal B_{n,X_s}$.  By \eqref{eq6}, we have ${\iota_s}_\ast F \in \mathcal B_{n,X}$, and so $F \cong \iota_s^\ast {\iota_s}_\ast F \in \mathcal B_{n,X}|_s$, giving us  (a).  In part (b), the first equality follows from part (a), while the second equality follows from Lemma \ref{lemma19}.  Therefore, $\mathcal T = \mathcal B_{n,X}$ is  an example of a torsion class in $\Coh (X)$ for which the answer to question (a) in Remark \ref{remark4} is `yes'.
\end{example}

\begin{remark}\label{remark21}
Suppose $\pi, \wh{\pi}$ satisfy conditions (i) and (ii) in the beginning of Section \ref{beginning},  that they both have relative dimension 1, and that the kernels of $\Psi$ and $\Phi$ are both sheaves sitting at degree 0.  Then by Lemma \ref{lemma17}, the category $\{\Coh^{\leq 0}(Y_s)\}^\uparrow$ is a torsion class in $\Coh (Y)$, and by  Lemma \ref{lemma20}, every $E \in \{\Coh^{\leq 0}(Y_s)\}^\uparrow$ is $\Phi$-WIT$_0$.  As a result, the category $\Phi (\{\Coh^{\leq 0}(Y_s)\}^\uparrow)$ (which will be used frequently later on) is contained in $W_{1,X}$.
\end{remark}

We briefly return to the question (a) in Remark \ref{remark4}.  Let us write $H^i$ to denote the $i$-th cohomology functor with respect to the standard t-structure on either $D(X)$ or $D(X_s)$, for any closed point $s \in S$.  Given a torsion class $\mathcal T$ in $\Coh (X)$, let us also write $\mathcal H^i$ to denote the $i$-th cohomology functor with respect to the t-structure on $D(X)$ with heart $\langle \mathcal T^\circ [1], \mathcal T\rangle$, or the $i$-th cohomology functor with respect to the t-structure on $D(X_s)$ with heart $\langle \mathcal (\mathcal T|_s)^\circ [1], \mathcal T|_s\rangle$, for any closed point $s \in S$.  Then, for any coherent sheaf $E$ on $X$, we have
\[
  E \in \mathcal T \text{ if and only if } \mathcal H^i (E)=0 \text{ for all $i \neq 0$},
\]
and for any closed point $s \in S$
\[
  E \in \{\mathcal T|_s\}^\uparrow \text{ if and only if } \mathcal H^i (E|_s)=0 \text{ for all $i \neq 0$}.
\]
The condition that $\{\mathcal T|_s\}^\uparrow \subseteq \mathcal T$, which is equivalent to having a `yes' to question (a) in Remark \ref{remark4}, is now equivalent to the condition
\begin{itemize}
  \item[] For any coherent sheaf $E \in \Coh (X)$, if $\mathcal H^i(E|_s)=0$ for all $i \neq 0$ and all closed points $s \in S$, then $\mathcal H^i (E)=0$,
\end{itemize}
which can be thought of as a `Nakayama's Lemma-type' statement.

The following observation on 1-dimensional closed subschemes of the total space of a fibration will be used from time to time:

\begin{lemma}\label{lemma75}
Let $\pi : X \to S$ be a proper morphism of varieties of relative dimension 1.  Let $Z$ be any irreducible, 1-dimensional closed subscheme of $X$.  Then $Z$ is either of the following two types:
 \begin{itemize}
 \item[(i)] $Z$ is contained in a fiber of $\pi$;
 \item[(ii)] for any $s \in S$, the intersection $Z \cap \pi^{-1}(s)$ is 0-dimensional if nonempty.
 \end{itemize}
\end{lemma}

\begin{proof}
Consider the locus
\[
  D := \{ s \in S : Z \cap \pi^{-1}(s) \text{ is 1-dimensional}\}.
\]
If $D$ is empty, then $Z$ is of type (ii).  Therefore, let us suppose $D$ is nonempty.  Then $\pi^{-1}(D)$ is a closed subset of $X$ by semicontinuity.  Note that, the dimension of $D$ must be exactly zero, or else $Z$ would have dimension at least 2, a contradiction.  That is, $D$ is a finite number of closed points.  Now, the intersection $\pi^{-1}(D) \cap Z$ is a 1-dimensional closed subset of $Z$.  By the irreducibility of $Z$, we have $\pi^{-1}(D) \cap Z = Z$, and $D$ consists of a single point, i.e.\ $Z$ is of type (i).
\end{proof}

\subsection{More torsion classes}\label{section-moretc}

In this section, we introduce a few more torsion classes in $\Coh (X)$ that depend on the geometry of the fibration.  Suppose $\pi : X \to S$ and $\hat{\pi}: Y \to S$ are a pair of dual elliptic fibrations.  We define
\begin{align}
  W_{0,X}' &:= \{ E \in \Coh (X) : E|_s \text{ is $\Psi_s$-WIT$_0$ for a general closed point $s \in S$}\}, \notag\\
   W_{1,X}' &:= \{ E \in \Coh (X) : E|_s \text{ is $\Psi_s$-WIT$_1$ for a general closed point $s \in S$}\}, \notag\\
   \mathfrak T_X &:= \langle W_{0,X}', (\Phi ( \{\Coh^{\leq 0} (Y_s)\}^\uparrow))\rangle. \label{eq23}
\end{align}
Note that, by Lemma \ref{lemma15}, for any closed point $s \in S$ and any coherent sheaf $F$ on $X_s$, the functor $\Psi_s$ takes $F$ to a complex of length at most 2, sitting at degrees 0 and 1.  Also,  we have defined $W_{0,X}', W_{1,X}'$ so that they contain sheaves that restrict to zero on a general fiber of $\pi$.  That is, we have $\mathcal B_X \subseteq W_{i,X}'$ for $i=0,1$.  In addition, we  have $\mathcal T_X \subseteq W_{0,X}'$, i.e.\ $W_{0,X}'$ contains all the torsion sheaves on $X$ - this is because for a torsion sheaf $T$ on $X$ and a general closed point $s \in S$, the restriction $T|_s$ must be 0-dimensional, which is  $\Psi_s$-WIT$_0$.

We can similarly define $W_{0,Y}', W_{1,Y}'$ and $\mathfrak T_Y$, with $X$ replaced with $Y$ and $\Psi$ replaced with $\Phi$.

\begin{lemma}\label{lemma47}
We have
 \begin{itemize}
 \item[(i)] $W_{0,X} \subseteq W_{0,X}' \subseteq \mathfrak T_X$;
 \item[(ii)] $\mathfrak T_X^\circ \subseteq (W_{0,X})^\circ = W_{1,X} \subseteq W_{1,X}'$.
\end{itemize}
\end{lemma}

\begin{proof}
Part (i) follows immediately  from Lemma \ref{lemma20} and the definition of $\mathfrak T_X$.

In part (ii), the first inclusion follows from part (i), while the second inclusion follows from $(W_{0,X}, W_{1,X})$ being a torsion pair in $\Coh (X)$ (see, for instance, \cite[Lemma 9.2]{BMef}).  To show the last inclusion of part (ii), take any $E \in W_{1,X}$.  By generic flatness, there exists a dense open subscheme $U \subseteq S$ such that  both $E|_U$ and $(\wh{E})|_U$ are flat.  By base change \cite[Proposition 6.1]{FMNT}, the restriction $E|_U$ is $\Psi_U$-WIT$_1$.  Then by \cite[Corollary 6.3(iii)]{FMNT}, we have that $E|_s$ is $\Psi_s$-WIT$_1$ for every closed point $s \in U$, i.e.\ $E \in W_{1,X}'$.
\end{proof}

\begin{remark}\label{remark18}
Since $(W_{0,X}, W_{1,X})$ is a torsion pair in $\Coh(X)$, every coherent sheaf $E$ on $X$ fits in a short exact sequence in $\Coh (X)$
\begin{equation}\label{eq9}
0 \to E_0 \to E \to E_1 \to 0
\end{equation}
where $E_0 \in W_{0,X}$ and $E_1 \in W_{1,X}$.  By Lemma \ref{lemma47}, we can also regard $E_0, E_1$ as objects in $W_{0,X}', W_{1,X}'$, respectively.
\end{remark}

\begin{lemma}\label{lemma29}
 For any $E \in \Coh (X)$, let $E_0, E_1$ be as in \eqref{eq9}.  Then:
\begin{itemize}
\item[(i)] If $E \in W_{i,X}'$ (where $i=0$ or $1$), then $E_{1-i} \in \mathcal B_X$.
\item[(ii)] $W_{1,X}' \cap \mathcal B_X^\circ \subseteq W_{1,X}$.
\end{itemize}
\end{lemma}


\begin{proof}
For part (i), consider the short exact sequence \eqref{eq9}:
\begin{equation}\label{eq14}
0 \to E_0 \overset{\al}{\to} E \to E_1 \to 0.
\end{equation}
For any closed point $s \in S$,  we have the exact sequence
\begin{equation}\label{eq58}
E_0|_s \overset{\al_s}{\to} E|_s \to E_1|_s \to 0.
\end{equation}

To begin with, suppose $E \in W_{0,X}'$.  Then for a general $s$, the restriction $E_1|_s$ is $\Psi_s$-WIT$_0$.  On the other hand, the base change \eqref{eq7} gives  $(\wh{E_1})|_s \cong \Psi_s^1 (E_1|_s)$.  Thus $(\wh{E_1})|_s=0$ for a general $s \in S$, i.e.\  $\wh{E_1} \in \mathcal B_X$.

Next, suppose $E \in W_{1,X}'$.  By Lemma \ref{lemma20},  the restriction $E_0|_s$ is  $\Psi_s$-WIT$_0$ for every closed point $s \in S$.  However, $E|_s$ is $\Psi_s$-WIT$_1$ for a general closed point $s \in S$ by assumption.  Therefore,   the map $\al |_s$ in \eqref{eq58} must be zero for a general closed point $s \in S$.  In other words, the injection $\al$ in $\Coh (X)$ vanishes when we restrict to a general fiber over $S$.  Hence $E_0 \in \mathcal B_X$, and the lemma is proved.

For part (ii), take any $E \in W_{1,X}' \cap \mathcal B_X^\circ$.  Let $E_0, E_1$ be as in \eqref{eq9}.  From part (i), we know that $E_0 \in \mathcal B_X$, and hence must vanish, implying $E \in W_{1,X}$.
\end{proof}

\begin{lemma}\label{lemma31}
We have
\begin{itemize}
\item[(i)] The category $W_{0,X}'$ is closed under quotients and extensions.
\item[(ii)] $W_{1,X}' \cap \mathcal F_X = ( W_{0,X}')^\circ$.
\item[(iii)] The category $\mathfrak T_X$ is closed under quotients and extensions in $\Coh (X)$.
\end{itemize}
\end{lemma}

\begin{proof}
Part (i) is clear.  For part (ii), let us first show that $W_{1,X}' \cap \mathcal F_X \subseteq (W_{0,X}')^\circ$.  Take any morphism $\al : F \to E$ in $\Coh (X)$ where $E \in W_{1,X}' \cap \mathcal F_X, F \in W_{0,X}'$.   We want to show that $\al$ is the zero map.  Since $W_{0,X}'$ is closed under quotients by part (i), we can assume that $\al$ is an injection in $\Coh (X)$.  Since $E$ is torsion-free, we can also assume that $F$ is torsion-free and nonzero.  Then, for a general closed point $s \in S$, $F|_s$ is $\Psi_s$-WIT$_0$ while $E|_s$ is $\Psi_s$-WIT$_1$, implying $\al|_s=0$ for a general $s$.  Thus $F$ must be torsion, a contradiction.  Hence $\al$ must be zero.

For the other inclusion in (ii), take any $E \in (W_{0,X}')^\circ$ that is nonzero. By Remark \ref{remark18}, we have $E \in W_{1,X}'$.  Since $W_{0,X}'$ contains all the torsion sheaves on $X$, we also know $E$ is torsion-free.  Thus (ii) is proved.

Given part (i), in order to show part (iii), it is enough to show that any quotient of an object $F$ in $\Phi ( \{\Coh^{\leq 0} (Y_s)\}^\uparrow)$ lies in $\mathfrak T_X$.  Consider any surjection $F \overset{\alpha}{\twoheadrightarrow} F'$  in $\Coh (X)$.  We can assume that $F'$  lies in $(W_{0,X}')^\circ$.  Therefore, from part (ii), we know $F'$ is torsion-free and lies in $W_{1,X}'$. Now, consider the short exact sequence in $\Coh (X)$
\[
0 \to K \to F \overset{\alpha}{\to} F' \to 0
\]
where $K = \kernel (\alpha)$.  From Lemma \ref{lemma29}(ii), we also know $F' \in W_{1,X}$.  On the other hand, that $F \in W_{1,X}$ implies $K \in W_{1,X}$.  The short exact sequence above is therefore taken by $\Psi$  to the short exact sequence in $\Coh (Y)$
\[
0 \to \wh{K} \to \wh{F} \to \wh{F'} \to 0.
\]
Since $\wh{F} \in \{ \Coh^{\leq 0}(Y_s)\}^\uparrow$, the same holds for $\wh{F'}$, and so  $F' \in \Phi (\{\Coh^{\leq 0}(Y_s)\}^\uparrow) \subseteq \mathfrak T_X$.    Thus (iii) holds. 
\end{proof}

By Lemma \ref{lemma31}(iii), we now have a torsion pair $(\mathfrak T_X, \mathfrak T_X^\circ)$ in $\Coh (X)$.  We set
\[
   \mathfrak C_X := \langle \mathfrak T_X^\circ [1], \mathfrak T_X \rangle.
\]

For any subcategory $\mathcal C$ of $\Coh (X)$, where $X$ is a smooth projective variety, we will define $\mathcal C^D$ to be the subcategory of $\Coh (X)$ consistant of all sheaves of the form $\EExt^c_X (F,\OO_X)$, where $F \in \mathcal C$ and $c$ is the codimension of the support of $F$.

Now we define, for an elliptic fibration $\wh{\pi}$ of any dimension $n$ and when $Y$ is smooth,

\begin{equation}\label{eq64}
  \mathcal B_{Y,\ast} := \langle \Coh^{\leq 0}(Y), (W_{1,Y}\cap \Coh^{=1}(Y) )^D \rangle.
\end{equation}

\begin{remark}
In \cite[Definition 1.1.7]{HL}, for a coherent sheaf $E$ of codimension $c$ on a smooth projective variety $X$, the notation $E^D$ denotes the sheaf $\EExt^c_{X}(E,\omega_X)$ where $\omega_X$ is the canonical sheaf on $X$.
\end{remark}

\begin{lemma}\label{lemma30}
The category $\mathcal B_{Y,\ast}$ is closed under quotients and extensions in $\Coh (Y)$.
\end{lemma}

\begin{proof}
Since $\mathcal B_{Y,\ast}$ is defined to be the category generated by $\Coh^{\leq 0}(Y)$ and $(W_{1,Y}\cap \Coh^{=1}(Y))^D$ via extensions, we only need to verify that it is  closed under quotients.

Take any nonzero surjection $\al : E \twoheadrightarrow T$ in $\Coh (Y)$, where $E \in (W_{1,Y} \cap \Coh^{=1}(Y))^D$.  We want to show that $T \in \mathcal B_{Y,\ast}$ as well.  Note that, it suffices to assume $T$ is pure 1-dimensional.  By definition, $E \cong \EExt^{n-1} (F,\OO_Y)$ for some $F \in W_{1,Y}\cap \Coh^{=1}(Y)$.  That  $F$ is  pure 1-dimensional implies $E \cong F^\vee [n-1]$.  Hence we have a short exact sequence in $\Coh (Y)$
\begin{equation}\label{eq16}
0 \to K \to F^\vee [n-1] \overset{\al}{\to} T \to 0
\end{equation}
where $K := \kernel (\al)$.  Assuming that $\al$ is not an isomorphism, we have that $K$ is a pure 1-dimensional sheaf.  That is, all the terms in \eqref{eq16} are pure 1-dimensional sheaves.  Considering \eqref{eq16} as an exact triangle in $D(Y)$, then dualising and shifting, we obtain the exact triangle in $D(Y)$
\begin{equation}\label{eq17}
  T^\vee [n-1] \overset{\al^\vee [n-1]}{\longrightarrow} F \to K^\vee [n-1] \to T^\vee [n]
\end{equation}
where all the terms $T^\vee [n-1], F$ and $K^\vee [n-1]$ are again pure 1-dimensional sheaves.  As a result, we have a short exact sequence in $\Coh (Y)$
\[
0 \to  T^\vee [n-1] \overset{\al^\vee [n-1]}{\longrightarrow} F \to K^\vee [n-1] \to 0.
\]
  Thus $T^\vee [n-1] \in W_{1,Y} \cap \Coh^{=1}(Y)$, and $T \cong \EExt^{n-1}_Y (T',\OO_Y)$ where $T' := T^\vee [n-1]$.  Hence $\mathcal B_{Y,\ast}$ is closed under quotients, and we are done.
\end{proof}

\begin{remark}\label{remark5}
Since $\Coh (Y)$ is a Noetherian abelian category, we  have another torsion pair $(\mathcal B_{Y,\ast}, (\mathcal B_{Y,\ast})^\circ)$ in $\Coh (Y)$ by \cite[Lemma 1.1.3]{Pol}.
\end{remark}
Let us define, for an elliptic fibration $\wh{\pi} : Y \to S$ of any dimension,
\[
  \mathfrak D_Y := \langle \mathcal B_{Y,\ast}^\circ [1], \mathcal B_{Y,\ast}\rangle.
  \]

\begin{remark}\label{remark8}
When $Y$ is  an elliptic surface, the objects of $\mathcal B_Y$ are exactly the fiber sheaves.  Also, since 0-dimensional sheaves on $Y$ are always $\Phi$-WIT$_0$, the objects of $W_{1,Y} \cap \mathcal B_Y$ are exactly the pure 1-dimensional $\Phi$-WIT$_1$ fiber sheaves in this case.  Therefore, we have the following equivalent description of $\mathcal B_{Y,\ast}$ when $Y$ is an elliptic surface:
\begin{equation}\label{eq20}
  \mathcal B_{Y,\ast} = \langle \Coh^{\leq 0}(Y), (W_{1,Y} \cap \mathcal B_Y)^D \rangle.
\end{equation}
\end{remark}

\begin{remark}\label{remark11}
On the other hand, on an elliptic fibration of any dimension, if $E$ is a $\Phi$-WIT$_1$ pure 1-dimensional sheaf, then $E$ cannot have any subsheaf lying in $\{ \Coh^{\leq 0}(Y_s)\}^\uparrow$, which is contained in $W_{0,Y}$.  That is, a $\Phi$-WIT$_1$ pure 1-dimensional sheaf on an elliptic fibration of any dimension is a fiber sheaf.
\end{remark}

\section{Elliptic fibrations}\label{section-equivofhearts}

In this section, we will consider a pair of dual elliptic fibrations $\pi : X \to S$ and $\wh{\pi} : Y \to S$.

We first prove for elliptic surfaces, that the t-structure on $D(X)$ with heart $\mathfrak T_X$ is equivalent to the t-structure on $D(Y)$ with heart $\mathfrak D_Y$ (up to a shift) via a derived equivalence from $D(X)$ to $D(Y)$ - see Theorem \ref{theorem1}.  This  result is a special case of Yoshioka's result  \cite[Proposition 3.3.5]{YosPII}.  We then extend the  above result to the case of elliptic threefolds -  see Theorem \ref{theorem3} and Theorem \ref{theorem2} for the precise statements;  in this case, the two t-structures differ by a tilt (in the sense of \cite[Chap.\ I, Sec.\ 2]{HRS}).  Below, we choose to discuss  elliptic surfaces and elliptic threefolds separately because the two cases are interesting in their own rights.

Our central idea is to filter  coherent sheaves on $X$ or $Y$ using the following torsion classes in $\Coh (X)$ (some of which are Serre subcategories) and their counterparts in $\Coh (Y)$, to the point that we understand the image under the Fourier-Mukai transform $\Psi$ of any subfactor in the filtration:
\begin{equation}\label{eq63}
 W_{0,X}, W_{0,X}', \{\Coh^{\leq 0}(X_s)\}^\uparrow, \mathfrak T_X, \mathcal B_{X,\ast}, \Coh^{\leq d}(X), \Coh (\pi)_{\leq d}
\end{equation}
where $d \geq 0$.

In Yoshioka's work \cite{YosAS}, he considered an elliptic surface $\pi : X \to S$ with a zero section $\sigma$, where all the fibers of $\pi$ are integral.  After identifying a compactification $\wh{\pi} : Y \to S$ of the relative Jacobian with $\pi$ itself and using the Poincar\'{e} sheaf as the kernel, he proceeded to consider a Fourier-Mukai transform $\Psi : D(X) \to D(X)$.  In \cite[Theorem 3.15]{YosAS} and \cite[Remark 3.5]{YosAS}, Yoshioka proved an isomorphism between two moduli spaces of semistable sheaves on $X$ where:
\begin{itemize}
\item one of the two moduli spaces parametrises pure 1-dimensional sheaves (so they have rank zero);
\item the semistability is with respect to $\sigma +kf, k \gg 0$ where $f$ is a fiber class for the fibration $\pi$;
\item the isomorphism is induced by the composite functor $(\Psi (-))^\vee$, i.e.\  the Fourier-Mukai transform $\Psi$ followed by the dualising functor on $D(X)$.
\end{itemize}
Later, in \cite[Proposition 3.4.5]{YosPII}, he generalised the above results to the case of twisted semistable perverse coherent sheaves on dual elliptic surfaces that arise as resolutions of singularities.

Following Yoshioka's idea, we will study the functor that is the composition of the Fourier-Mukai transform from an elliptic fibration to its dual, followed by the dualising functor.  From now on, we will write $\Lambda (-)$ to denote the composite functor $(\Psi (-))^\vee$, i.e.\ the Fourier-Mukai functor $\Psi (-) : D(X) \to D(Y)$ followed by the derived dual $-^\vee : D(Y) \to D(Y)$, irrespective of the dimensions of $X$ and $Y$.  We will also write $\Lambda^i (-)$ to denote $H^i (\Lambda (-))$, where $H^i (-)$ is the degree-$i$ cohomology functor with respect to the standard t-structure on $D(Y)$.

\subsection{t-structures on elliptic surfaces}\label{sec-ellipsurf}

In this section,  we assume that $\pi : X \to S$ and $\wh{\pi} : Y \to S$ are a pair of dual elliptic surfaces.

\begin{lemma}\label{lemma23}
Suppose $X$ is an elliptic surface, and $E \in \mathcal B_X$.  Let $E_0, E_1$ be as in \eqref{eq9}.  Then $\Lambda E \in D^{[0,1]}_{\Coh (Y)}$, and
\begin{itemize}
\item[(i)] $\Lambda^0 E \cong \EExt^1 (\widehat{E_1}, \OO_Y)$;
\item[(ii)] there is a short exact sequence in $\Coh (Y)$
\begin{equation}
0 \to \EExt^2 (\widehat{E_1}, \OO_Y) \to \Lambda^1E \to \EExt^1 (\widehat{E_0}, \OO_Y) \to 0.
\end{equation}
\end{itemize}
\end{lemma}



\begin{proof}
Take any $E \in \mathcal B_X$.  From \eqref{eq9}, we obtain the exact triangle in $D(Y)$
\begin{equation}\label{eq11}
\widehat{E_0} \to \Psi (E) \to \widehat{E_1}[-1] \to \widehat{E_0}[1].
\end{equation}
Taking derived duals, we obtain the exact triangle
\begin{equation}\label{eq10}
 (\widehat{E_1})^\vee [1] \to \Lambda E \to (\widehat{E_0})^\vee \to (\widehat{E_1})^\vee [2].
\end{equation}
Since $\widehat{E_0}$ is $\Phi$-WIT$_1$, it has no 0-dimensional subsheaves. Besides, since $E_0$ is a fiber sheaf, its transform $\widehat{E_0}$ remains a fiber sheaf.  Hence $\widehat{E_0}$ is pure 1-dimensional, and thus $\EExt^i (\widehat{E_0},\OO_Y)=0$ for all $i \neq 1$, meaning $(\widehat{E_0})^\vee \cong \EExt^1 (\widehat{E_0},\OO_Y)[-1]$ is a 1-dimensional sheaf sitting at degree 1.

On the other hand, since $E_1$ is a fiber sheaf, the same holds for $\wh{E_1}$, and so $\EExt^0 (\wh{E_1},\OO_Y)=0$.  The lemma then follows by taking the long exact sequence of cohomology of \eqref{eq10}.
\end{proof}

\begin{lemma}\label{lemma22}
Suppose $X$ is an elliptic surface, and $E \in \mathcal T_X \cap \mathcal B_X^\circ$.  Then $\Lambda E \in D^{[0,1]}_{\Coh (Y)}$.  Furthermore:
\begin{itemize}
\item[(i)] If $E \in W_{0,X}$,  then $\wh{E}$ is a locally free sheaf.
\item[(ii)]  $\Lambda^1 E$ is a 0-dimensional sheaf that is a quotient of $\EExt^2 (\wh{E_1},\OO_Y)$, where $E_1$ is as in \eqref{eq9}.
\end{itemize}
\end{lemma}


\begin{proof}
Take any $E \in \mathcal T_X \cap \mathcal B_X^\circ$, then $E$ has no fiber subsheaves, and in particular, is pure 1-dimensional.  Let $E_i$ be as in \eqref{eq9}.  By \cite[Corollary 5.4]{Lo5}, we know $\wh{E_0}$ is a locally free sheaf.  Thus part (i) holds.

By \cite[Lemma 2.6]{Lo5}, we know $E_1$ is a fiber sheaf, and so $(\wh{E_1})^\vee$ is a 2-term complex sitting at degrees 1 and 2, where the degree-2 cohomology is $\EExt^2 (\wh{E_1}, \OO_Y)$, which is 0-dimensional.  Since $E_0$ is a subsheaf of $E$, we have $E_0 \in \mathcal T_X \cap \mathcal B_X^\circ$.  And by part (i), we know $\widehat{E_0}$ is locally free, and so $\EExt^1 (\wh{E_0},\OO_Y)=0$.  Taking the long exact sequence of cohomology of \eqref{eq10} then gives us part (ii) of the lemma.  That $\Lambda E \in D^{[0,1]}_{\Coh (Y)}$ also follows from the long exact sequence.
\end{proof}

Before we consider the images of sheaves supported in dimension 2 under the functor $\Lambda$, we prove:

\begin{lemma}\label{lemma26}
Suppose $\pi : X \to S$ is an elliptic surface or an elliptic threefold.  Suppose $E$ is a pure d-dimensional, $\Psi$-WIT$_0$ sheaf on $X$.
 \begin{itemize}
 \item[(i)] If $E \in \Coh(\pi)_{d-1}$, then $\wh{E}$ is a pure sheaf of dimension $d$.
 \item[(ii)] If $E \in \Coh(\pi)_d$ and $E$ has no subsheaves $E'$ in $\Coh(\pi)_{d-1}$,  then $\wh{E}$ is a pure sheaf of dimension $d+1$.
 \end{itemize}
\end{lemma}


\begin{proof}
(i):  suppose $E$ is a pure $d$-dimensional $\Psi$-WIT$_0$ sheaf lying in $\Coh (\pi)_{d-1}$.  When $d=3$, $E$ is torsion-free, and the result is just \cite[Lemma 9.4]{BMef}.  When $d=1$, $E$ is a fiber sheaf, and so $\wh{E}$ is a $\Phi$-WIT$_1$ fiber sheaf, which is necessarily pure.  Now, suppose $d=2$.  Then $\wh{E} \in \Coh (\hat{\pi})_{1}$ by Lemma \ref{lemma66}, and so $\dimension \wh{E} \leq 2$.  Suppose there is a nonzero subsheaf $T$ of $\wh{E}$ where $T \in \Coh_{\leq 1}(Y)$.  Since $\wh{E} \in W_{1,Y}$, we have $T \in W_{1,Y}$ as well.  In particular,  $T$ cannot have any subsheaf in $\{ \Coh^{\leq 0}(Y_s)\}^\uparrow$ by Lemma \ref{lemma20}.  Hence $T$ is forced to be a fiber sheaf.  The injection $T \hookrightarrow \wh{E}$ is then transformed under $\Phi$ to a nonzero map $\wh{T} \to E$, implying $E$ has a fiber subsheaf, contradicting its purity.  Hence $\wh{E}$ must be pure when $d=2$. This finishes the proof of part (i).

(ii):  suppose $E$ is a pure $d$-dimensional $\Psi$-WIT$_0$ sheaf, except that now we suppose  $E \in \Coh (\pi)_d$.  Let us write $Z := \hat{\pi} (\text{supp}\, \wh{E}) = \pi (\text{supp}\, E)$.  In this case, the fiber $E|_s$ is 0-dimensional for a general closed point $s \in Z$.  We will now show that, for a general closed point $s \in Z$, we have $\dimension (\wh{E}|_s)=1$.  If we write $\wh{E}_{Z_{red}}$ for the pullback of $\wh{E}$ along the base change $Z_{red} \hookrightarrow Z \hookrightarrow S$, then it is enough to show $\dimension ((\wh{E}|_{Z_{red}})|_s)=1$ for a general $s \in Z$.  That is, we can assume that $Z$ is reduced.  Applying generic flatness to $E$ and $\wh{E}$ with respect to the morphism $X \times_S Z \to Z$ together with \cite[Corollary 6.3(3)]{FMNT}, we obtain that, for a general $s \in Z$, the restriction $\wh{E}|_s \cong \wh{E|_s}$ is 1-dimensional, as wanted.  Therefore, we have $\dimension \wh{E}=d+1$.


Now, suppose we have an injection $T \hookrightarrow \wh{E}$ where $0\neq T \in \Coh_{\leq d}(Y)$.  Then $T$ is $\Phi$-WIT$_1$.  We consider the different cases:
\begin{itemize}
\item When $d=2$, and $X$ is of dimension 3: if $\hat{\pi}(\supp (T))$ is $2$-dimensional (and hence equal to $S$), then $T$ itself is $2$-dimensional.  This implies that $T|_s$ is 0-dimensional for a general closed point $s \in S$, and so $T \in W_{0,Y}'$.  However, we also have $T \in W_{1,Y}'$ by Lemma \ref{lemma47}(ii).  Thus $T$ lies in $\in W_{0,Y}' \cap W_{1,Y}'$, which is contained in $\mathcal{B}_Y$ by Lemma \ref{lemma29}, i.e.\ $\dimension (\wh{\pi}(\supp (T))) \leq 1$.  The injection $T \hookrightarrow \wh{E}$ is then taken by $\Phi$ to a nonzero morphism $\wh{T} \to E$, contradicting the assumption that $E$ has no subsheaves in $\Coh (\pi)_{d-1}$.
\item When $d=1$, and $X$ is of dimension 2 or 3: by Lemma \ref{lemma75} and Remark \ref{remark21}, that $T \in W_{1,Y}$ implies $T$ cannot have any subsheaf in $\{\Coh^{\leq 0}(Y_s)\}^\uparrow$ and must be a fiber 1-dimensional sheaf.  Then the injection $T \hookrightarrow \wh{E}$ is taken by $\Phi$ to a nonzero morphism $\wh{T} \to E$, contradicting the assumption that $E$ has no subsheaves lying in $\Coh (\pi)_{d-1}$.
\end{itemize}
Hence $\wh{E}$ must be pure of dimension $d+1$, finishing the proof of (ii).
\end{proof}

Lemma \ref{lemma26} also yields the following results on reflexivity of sheaves under Fourier-Mukai transforms:

\begin{lemma}\label{lemma38}
Let $\pi : X \to S$ be an elliptic surface or an elliptic threefold.  Suppose $E$ is a $\Psi$-WIT$_0$ torsion-free  sheaf.  Then $\wh{E}$ is torsion-free, and is reflexive whenever
\[
\Ext^1 (W_{0,X} \cap \Coh (\pi)_{\leq 0} , E)=0.
\]
\end{lemma}

\begin{proof}
By Lemma \ref{lemma26}(i), we know $\wh{E}$ is pure of codimension 0, so we have a short exact sequence
\begin{equation}\label{eq21}
  0 \to \wh{E} \to (\wh{E})^{DD} \to T \to 0
\end{equation}
where $(\wh{E})^{DD}$, being the double dual of $\wh{E}$ (where the `dual' of a sheaf is in the sense of \cite[Definition 1.1.7]{HL}), is torsion-free and reflexive, while  $T$ is a coherent sheaf of codimension at least 2 (which implies $T \in \mathcal B_Y$).  Note that $\Ext^1 (W_{0,Y}\cap \mathcal B_Y,\wh{E}) \cong \Hom (W_{1,X}\cap \mathcal B_X, E)=0$ since $E$ is torsion-free.  Now, take any $A \in W_{0,Y} \cap \mathcal B_Y$.  Applying the functor $\Hom (A,-)$ to \eqref{eq21}, we get $\Hom (A,T)=0$.  In other words, we have $\Hom (W_{0,Y} \cap \mathcal B_Y, T)=0$.  This implies that $T \in W_{1,Y}$, and so $(\wh{E})^{DD}$ is also $\Phi$-WIT$_1$.  On the other hand, since $\dimension T \leq 1$, Lemma \ref{lemma75} and Remark \ref{remark21} together imply that $T$ must be a fiber sheaf.  That is, we have $T \in W_{1,Y} \cap \Coh (\wh{\pi})_{\leq 0}$.  The lemma then follows from
\[
  \Ext^1 (W_{1,Y}\cap \Coh (\wh{\pi})_{\leq 0},\wh{E}) \cong \Ext^1 (W_{0,X} \cap \Coh (\pi)_{\leq 0},E).
\]
\end{proof}

\begin{coro}\label{coro4}
Suppose $X$ is an elliptic threefold where all the fibers are Cohen-Macaulay with trivial dualising sheaves.  If $E$ is a $\Psi$-WIT$_0$ reflexive torsion-free sheaf, then $\wh{E}$ is also a reflexive torsion-free sheaf.
\end{coro}

\begin{proof}
Since $E$ is reflexive and torsion-free, we have $\Ext^1 (\Coh^{\leq 1}(X),E)=0$ by \cite[Lemma 4.21]{Lo10}.  The corollary then follows from Lemma \ref{lemma38}.
\end{proof}

\begin{lemma}\label{lemma70}
Let $X$ be an elliptic threefold.  Suppose that:
 \begin{itemize}
 \item $E \in W_{0,X}$;
 \item $\Hom (\Phi (\Coh^{\leq 0}(Y)),E)=0$; and
 \item $\wh{E}$ is a pure sheaf of dimension at least 2.
 \end{itemize}
 Then $\EExt^2 (\wh{E},\OO_Y)=0$, and $\Lambda (E) \in D^{[0,1]}_{\Coh (Y)}$. In particular, if $\wh{E}$ is pure of dimension 2, then $\wh{E}$ is reflexive.
\end{lemma}

\begin{proof}
Consider the canonical short exact sequence
\begin{equation}\label{eq56}
0 \to \wh{E} \to (\wh{E})^{DD} \to T \to 0
\end{equation}
in $\Coh (Y)$ where $T \in \Coh^{\leq 1}(Y)$.  Since $(\wh{E})^{DD}$ is a reflexive sheaf on a threefold, we have $\EExt^i((\wh{E})^{DD},\OO_Y)=0$ for $i \geq 2$ regardless of whether $\wh{E}$ is of dimension 2 or 3 (see \cite[Proposition 1.1.6(ii), Proposition 1.1.10(4')]{HL}).  On the other hand,  from \eqref{eq56} we have the exact sequence
\[
\EExt^2( (\wh{E})^{DD},\OO_Y) \to \EExt^2 (\wh{E},\OO_Y) \to \EExt^3 (T,\OO_Y) \to 0,
\]
which gives $\EExt^2 (\wh{E},\OO_Y) \cong \EExt^3 (T,\OO_Y)$.  We will now show that $\EExt^3 (T,\OO_Y)$ vanishes by showing that $T$ is pure 1-dimensional.

Let $Q$ be the maximal 0-dimensional subsheaf of $T$; we can pull back the short exact sequence \eqref{eq56} along the inclusion $Q \hookrightarrow T$, to obtain the following commutative diagram of short exact sequences in $\Coh (Y)$:
\[
\xymatrix{
  0 \ar[r] & \wh{E} \ar@{=}[d] \ar[r] & F \ar[d] \ar[r] & Q \ar[d] \ar[r] & 0 \\
  0 \ar[r] & \wh{E} \ar[r] & (\wh{E})^{DD} \ar[r] & T \ar[r] & 0
}.
\]
Then $F$ is necessarily pure of dimension at least 2, since it is a  subsheaf of $(\wh{E})^{DD}$.  However, we have $\Ext^1 (Q,\wh{E}) \cong \Hom (Q,\wh{E}[1]) \cong \Hom (\Phi (Q),E)$, which vanishes since $\Hom (\Phi (\Coh^{\leq 0}(Y)),E)=0$ by assumption, implying $F \cong \wh{E} \oplus Q$, contradicting the purity of $F$ unless $Q=0$.  Hence $T$ is pure 1-dimensional, and we obtain $\Lambda (E) \in D^{[0,1]}_{\Coh (Y)}$.  The last assertion of the lemma follows from \cite[Proposition 1.1.10]{HL}.
\end{proof}

\begin{coro}\label{coro3}
Suppose $\pi : X \to S$ is an elliptic threefold. Suppose $E$ lies in $\Coh^{\leq 1}(X) \cap \Coh (\pi)_1$ and has no  fiber subsheaves.  Then $E$ is $\Psi$-WIT$_0$, and its transform $\wh{E}$ is a 2-dimensional reflexive sheaf.
\end{coro}

\begin{proof}
By Lemma \ref{lemma75}, we have $E \in \Phi (\{\Coh^{\leq 0}(Y_s)\}^\uparrow)$, which implies $E$ is $\Psi$-WIT$_0$ by Remark \ref{remark21}.  That $E$ has no fiber subsheaves implies $\wh{E}$ is pure of dimension 2 by Lemma \ref{lemma26}(ii).  Lemma \ref{lemma70} gives us the reflexivity of $\wh{E}$.
\end{proof}

\begin{lemma}\label{lemma24}
Suppose $X$ is an elliptic surface, and $E \in \Coh^{=2} (X) \cap W_{0,X}'$.  Then $\Lambda E \in D^{[0,1]}_{\Coh (Y)}$, and $\Lambda^1 E$ is a 0-dimensional sheaf.
\end{lemma}

\begin{proof}
By Lemma \ref{lemma29}, we have $E_1 \in \mathcal B_X$.  On the other hand, $\wh{E_0}$ is pure 2-dimensional by Lemma \ref{lemma26}.  Hence in the exact triangle \eqref{eq10}, the complex $(\wh{E_0})^\vee$ lies in $D^{[0,1]}_{\Coh (Y)}$,  while $\Lambda E_1 = (\wh{E_1})^\vee [1]$ also lies in $D^{[0,1]}_{\Coh (Y)}$ by Lemma \ref{lemma23}.  As a result,  we have $\Lambda E \in D^{[0,1]}_{\Coh (Y)}$.  In particular, we have the exact sequence
\begin{equation}\label{eq12}
\EExt^2 (\wh{E_1}, \OO_Y) \to \Lambda^1 E \to \EExt^1 ( \wh{E_0}, \OO_Y) \to 0.
\end{equation}
Since $\wh{E_0}$ is pure 2-dimensional, the sheaf $\EExt^1 ( \wh{E_0}, \OO_Y)$ is 0-dimensional, as is $\EExt^2 (\wh{E_1}, \OO_Y)$.  Hence $\Lambda^1 E$ itself is 0-dimensional.
\end{proof}

\begin{lemma}\label{lemma32}
Let $X$ be an elliptic surface, and suppose $E \in \Phi (\{ \Coh^{\leq 0}(Y_s)\}^\uparrow )$.  Then
  \begin{itemize}
  \item[(i)] $\Lambda E \in D^{[0,1]}_{\Coh (Y)}$ and $\Lambda^1 E$ is a 0-dimensional sheaf;
   \item[(ii)] $E$ is torsion-free if and only if $\Lambda E \cong \EExt^1 (\wh{E},\OO_Y)$ is a pure 1-dimensional sheaf (lying in $\{\Coh^{\leq 0}(Y_s)\}^\uparrow$) sitting at degree 0.
   \end{itemize}
\end{lemma}

\begin{proof}
By Lemma \ref{lemma20}, the category $\Phi (\{ \Coh^{\leq 0}(Y_s)\}^\uparrow )$ is contained in $W_{1,X}$.  Take any  $E \in \Phi (\{ \Coh^{\leq 0}(Y_s)\}^\uparrow )$.  Then $\wh{E}$ is $\Phi$-WIT$_0$.  Consider the short exact sequence
\begin{equation}\label{eq19}
0 \to F_0 \to \wh{E} \to F_1 \to 0
 \end{equation}
 in $\Coh (Y)$ where $F_0$  is the maximal 0-dimensional subsheaf of $\wh{E}$.  Then both $F_0, F_1$ are $\Phi$-WIT$_0$, and the dimension of $F_1$ is 1 if $F_1 \neq 0$; we also have the exact triangle in $D(Y)$
\[
  F_1^\vee \to (\wh{E})^\vee \to F_0^\vee \to F_1^\vee [1].
\]
Here, $F_0^\vee$ is a 0-dimensional sheaf sitting at degree 2, while $F_1^\vee \cong \EExt^1 (F_1,\OO_Y)[-1]$.  On the other hand, we have $\Psi (E) \cong \wh{E}[-1]$, and so $\Lambda E \cong (\wh{E})^\vee [1] \in D^{[0,1]}_{\Coh (Y)}$.  Besides, the exact triangle above gives $\Lambda^1E \cong H^2 ( (\wh{E})^\vee) \cong H^2(F_0^\vee)$, which is a 0-dimensional sheaf.  Thus part (i) holds.

Now, the transform of \eqref{eq19} is a short exact sequence in $\Coh (X)$
\begin{equation}
0 \to \wh{F_0} \to E \to \wh{F_1} \to 0.
\end{equation}
Since $\wh{F_0}$ is a fiber sheaf, it must be zero when $E$ is torsion-free, in which case the argument in part (i) shows that $\Lambda E \cong \EExt^1 (\wh{E},\OO_Y)$ is a pure 1-dimensional sheaf sitting at degree 0.  

For the `if' part of part (ii), suppose that $\Lambda E \cong \EExt^1 (\wh{E},\OO_Y)$ is a pure 1-dimensional sheaf sitting at degree 0.  Then from the computation of part (i), we see that $F_0$ vanishes and $E \cong \wh{F_1}$. However, if $\wh{F_1}$  had a torsion subsheaf $F'$, then it would be a $\Psi$-WIT$_1$ fiber sheaf by \cite[Lemma 2.6]{Lo5}.  This means that $F_1$ itself would also have a fiber subsheaf, contradicting that $F_1$ is a pure 1-dimensional sheaf lying in $\{\Coh^{\leq 0}(Y_s)\}^\uparrow$.  Thus part (ii) is proved.
\end{proof}

\begin{lemma}\label{lemma28}
 Suppose $X$ is an elliptic surface and $E \in \Coh (X)$ is a pure 2-dimensional sheaf such that:
  \begin{itemize}
  \item $E \in W_{1,X}'$;
  \item $\Hom (\Phi (\{ \Coh^{\leq 0}(Y_s)\}^\uparrow ), E)=0$.
  \end{itemize}
  Then $\Lambda (E[1]) \in D^{[0,1]}_{\Coh (Y)}$  and $\Lambda^1 (E[1])$ sits in an exact sequence in $\Coh (Y)$ of the form
\[
  \EExt^1 (A,\OO_Y) \to \Lambda^1 (E[1]) \to \EExt^1 (B,\OO_Y) \to 0
\]
where $A$ is pure 2-dimensional (hence $\EExt^1 (A,\OO_Y)$ is 0-dimensional), and $B \in W_{1,Y}\cap \mathcal B_Y$ (in particular, $B$ is a pure 1-dimensional fiber sheaf).
\end{lemma}

\begin{proof}
Let $E_0, E_1$ be as in  \eqref{eq9}.  By Lemma \ref{lemma29}, we know $E_0 \in \mathcal B_X$.  However, we are assuming $E$ to be pure, and so $E=E_1$,  and $\Lambda (E[1]) \cong (\wh{E_1})^\vee$.

Suppose $T$ is the maximal torsion subsheaf of $\wh{E_1}$.  Note that, $\wh{E_1}$ has no 0-dimensional subsheaves, or else there would be a nonzero map from a fiber sheaf to $E_1$, contradicting the purity of $E$.  Therefore, $T$ must be a pure 1-dimensional sheaf if it is nonzero, and we have the short exact sequence in $\Coh (Y)$
\begin{equation*}
0 \to T \to \wh{E_1} \to \wh{E_1}/T \to 0
\end{equation*}
where $\wh{E_1}/T$ is a pure 2-dimensional sheaf.  We thus obtain an exact triangle in $D(Y)$
\[
  (\wh{E_1}/T)^\vee \to (\wh{E_1})^\vee \to T^\vee \to (\wh{E_1}/T)^\vee[1].
\]
Since $\wh{E_1}/T$ is pure 2-dimensional, the sheaf $\EExt^2 (\wh{E_1}/T,\OO_Y)$ vanishes, and so $(\wh{E_1}/T)^\vee$ has cohomology only in degrees 0 and 1.  Also, since $T$ is pure 1-dimensional, we have $\EExt^0 (T,\OO_Y)=0=\EExt^2 (T,\OO_Y)$.  Hence $T^\vee$ is a pure 1-dimensional sheaf sitting at degree 1.  Thus $\Lambda (E[1]) \cong (\wh{E_1})^\vee$ is a 2-term complex sitting in degrees 0 and 1, and we have the exact sequence in $\Coh (Y)$
\begin{equation}\label{eq15}
\EExt^1 (\wh{E_1}/T,\OO_Y) \to H^1 ( (\wh{E_1})^\vee) \to \EExt^1 (T,\OO_Y) \to 0.
\end{equation}
where $\EExt^1 (\wh{E_1}/T,\OO_Y)$ is a 0-dimensional sheaf.

To finish off the proof, observe that $T$ has no nonzero subsheaves lying in the category $\{\Coh^{\leq 0}(Y_s)\}^\uparrow$.  For, if $T$ had such a subsheaf $T'$, then the image of the composition $T' \hookrightarrow T \hookrightarrow \wh{E_1}=\wh{E}$ under $\Phi$ would give us a nonzero element in $\Hom (\Phi (\{ \Coh^{\leq 0}(Y_s)\}^\uparrow ), E)$, contradicting our assumption.  Thus $T$ must be a fiber sheaf by Lemma \ref{lemma75}.  Since $E=E_1$ is pure 2-dimensional, we have    $\Hom (\mathcal B_Y \cap W_{0,Y}, \wh{E_1})=0$, meaning the $\Phi$-WIT$_0$ part of $T$ vanishes, i.e.\ $T$ is $\Phi$-WIT$_1$ as claimed.
\end{proof}

Pulling the above results together, we can now characterise the image of the heart $\mathfrak C_X$ under the functor $\Lambda$:

\begin{proposition}\label{pro1}
Let $X$ be an elliptic surface.  Then for any $E \in \mathfrak C_X$, we have:
\begin{itemize}
\item[(i)] $\Lambda E \in D^{[0,1]}_{\Coh (Y)}$;
\item[(ii)] $\Lambda^1 E \in \mathcal B_{Y,\ast}$.
\end{itemize}
\end{proposition}

\begin{proof}
We have the following inclusions of torsion classes in $\Coh (X)$:
\[
  \mathcal B_X \subseteq \mathcal T_X = \Coh^{\leq 1}(X) \subseteq W_{0,X}'.
\]
Given any $E \in W_{0,X}'$, we can first find a short exact sequence in $\Coh (X)$
\[
 0 \to E^0 \to E \to E^1 \to 0
\]
where $E^0 \in \Coh^{\leq 1}(X)$ and $E^1 \in \Coh^{=2}(X) \cap W_{0,X}'$,  and then another short exact sequence in $\Coh (X)$
\[
 0 \to E^{0,0} \to E^0 \to E^{0,1} \to 0
\]
where $E^{0,0} \in \mathcal \mathcal B_X$ and $E^{0,1} \in \Coh^{\leq 1}(X) \cap \mathcal B_X^\circ$.  Setting $E'' := E^{0,0}$ and $E' := E^{0}$, we obtain a filtration in $\Coh (X)$
\[
  E'' \subseteq E' \subseteq E
\]
where $E'' \in \mathcal B_X$, $E'/E'' \in  \mathcal T_X \cap \mathcal B_X^\circ$ and $E/E' \in W_{0,X}' \cap \Coh^{=2}(X)$.  Since $\mathfrak T_X$ is defined as the extension closure $\langle W_{0,X}', \Phi (\{\Coh^{\leq 0}(Y_s)\}^\uparrow)\rangle$, Lemmas \ref{lemma23}, \ref{lemma22}, \ref{lemma24} and \ref{lemma32} together imply $\Lambda (\mathfrak T_X) \subset D^{[0,1]}_{\Coh (Y)}$.

Now, by the definition of $\mathfrak T_X$ and Lemma \ref{lemma47}(ii), we see that any object in $(\mathfrak T_X)^\circ$ satisfies the hypotheses of Lemma \ref{lemma28}.  This gives us $\Lambda (\mathfrak T_X^\circ [1]) \subset D^{[0,1]}_{\Coh (Y)}$.  Together with the last paragraph, we now have $\Lambda (\mathfrak C_X) \subset D^{[0,1]}_{\Coh (Y)}$, i.e.\ we have part (i) of the proposition.  Part (ii) of the proposition follows from the computations in Lemmas \ref{lemma23}, \ref{lemma22}, \ref{lemma24}, \ref{lemma32} and \ref{lemma28}.
\end{proof}

We now have the following theorem:

\begin{theorem}\label{theorem1}
When $X$ is a smooth elliptic surface,  the functor $\Lambda (-) := (\Psi (-))^\vee$  induces an equivalence between the t-structure $(D^{\leq 0}_{\mathfrak C}, D^{\geq 0}_{\mathfrak C})$ on $D(X)$, and the t-structure $(D^{\leq 0}_{\mathfrak D}, D^{\geq 0}_{\mathfrak D})$ on $D(Y)$.  Equivalently, $\Lambda$ induces an equivalence of hearts

\begin{equation}\label{eq18}
   \mathfrak C_X \overset{\thicksim}{\to} \mathfrak D_Y [-1].
\end{equation}
\end{theorem}

\begin{proof}
By Remark \ref{remark5} and Lemma \ref{lemma31}, the two categories $\mathfrak C_X, \mathfrak D_Y$ are both hearts of t-structures.
Proposition \ref{pro1}(i) shows that the functor $\Lambda$ takes $\mathfrak C_X$ to a heart of the form  $\langle \mathcal F[1], \mathcal T\rangle [-1]$, for some torsion pair $(\mathcal T, \mathcal F)$  in $\Coh (Y)$.  Moreover, Proposition \ref{pro1}(ii) shows that $\mathcal T \subseteq \mathcal B_{Y,\ast}$.  Therefore, to prove the theorem, it remains to show that $\mathcal B_{Y,\ast}\subseteq \mathcal T$. That is, it remains to show that every object $E \in \mathcal B_{Y,\ast}$ appears as the degree-1 cohomology of $\Lambda E'$ for some $E' \in \mathfrak C_X$.  Furthermore, from the construction of $\mathcal B_{Y,\ast}$ and Remark \ref{remark8}, it is enough to consider the following two cases:
\begin{itemize}
\item[(1)] when $E \cong \OO_y$ is a skyscraper sheaf of length 1 supported at a closed point $y \in Y$;
\item[(2)] when $E \in (W_{1,Y} \cap \mathcal B_Y)^D$.
\end{itemize}
In case (1), observe that $\wh{\OO_y} \in \mathcal B_X \subset \mathfrak C_X$, and $\Lambda (\wh{\OO_y}) \cong (\OO_y [-1])^\vee \cong \OO_y^\vee [1] \cong \OO_y [-1]$.  This shows that $\OO_y \in \mathcal T$.

In case (2), suppose $E \in (W_{1,Y}\cap \mathcal B_Y)^D$.  Then $E \cong \EExt^1_Y (F,\OO_Y)$ for some  $F \in W_{1,Y} \cap \mathcal B_Y$.  Then $\wh{F} \in \mathcal B_X \subset \mathfrak C_X$, and $\Lambda \wh{F} \cong F^\vee \cong \EExt^1 (F,\OO_Y)[-1] = E[-1]$, showing $E \in\mathcal T$.  This completes the proof of the theorem.
\end{proof}

In \cite[Section 3.3]{YosPII}, Yoshioka considered a torsion pair $(\overline{\mathfrak{T}}, \overline{\mathfrak{F}})$ in a category of perverse sheaves.  When the category of perverse sheaves coincides with $\Coh (X)$, the torsion class $\overline{\mathfrak{T}}$ in $\Coh (X)$ is the category of objects $E\in\Coh (X)$ such that in its relative Harder-Narasimhan filtration with respect to $\pi$ (see, for instance, \cite[(3.5)]{YosPII}), all the subfactors satisfy the inequality $\mu_f \geq 0$.  Here, $\mu_f$ is a slope function defined as follows: writing $f$ to denote the fiber class of the morphism $\pi$, and for any $F \in \Coh (X)$, we set
\[
  \mu_f (F) :=
  \begin{cases}
  \frac{c_1(F)\cdot f}{\rank {F}} &\text{ if $\rank (F)>0$}, \\
  +\infty &\text{ if $\rank (F)=0$}.
  \end{cases}
\]

When $\pi : X \to S$ is a smooth elliptic surface with a section and integral fibers, it has a relative compactified Jacobian $\wh{\pi} : Y \to S$ that is a fine moduli space and a universal `Poincar\'{e}' sheaf. For instance, $\wh{\pi}$ can be taken to be the Altman-Kleiman compactification \cite[Remark 6.33]{FMNT}.  Under this setting, Theorem \ref{theorem1} can be considered as a special case (the `untwisted' case, and where there are no singularities to resolve) of \cite[Proposition 3.3.6]{YosPII}, in the following precise sense:

\begin{lemma}\label{lemma76}
Suppose $\pi : X \to S$ is a smooth elliptic surface with a section and integral fibers, and  $\wh{\pi} : Y \to S$ is a compactification of the relative Jacobian of $X$, and $\Psi : D(X) \to D(Y)$ is the relative Fourier-Mukai transform with the Poincar\'{e} sheaf as the kernel.  Then the torsion classes $\overline{\mathfrak{T}}$ and $\mathfrak{T}_X$ in $\Coh (X)$ coincide.
\end{lemma}

\begin{proof}
From our definition \eqref{eq23} of $\mathfrak T_X$, it is clear that $\mu_f (F) \geq 0$ for any $F \in \mathfrak T_X$.  Since $\mathfrak T_X$ is closed under quotients, the right-most subfactor in the relative HN filtration of $F$ must have $\mu_f \geq 0$.  Hence $\mu_f (F') \geq 0$ for all the subfactors $F'$ in the relative HN filtration of $F$, and so $\mathfrak T_X \subseteq \overline{\mathfrak{T}}$.

For the other inclusion, take any $E \in \overline{\mathfrak{T}}$.  Consider the short exact sequence in $\Coh (X)$
\[
  0 \to E' \to E \to E'' \to 0
\]
where $E' \in W_{0,X}'$ and $E'' \in (W_{0,X}')^\circ$.  Then $E''$ is torsion-free and $\Phi$-WIT$_1$.  Note that the left-most subfactor in the relative HN filtration of $E''$ must have $\mu_f  \leq 0$, for otherwise that subfactor would lie in $W_{0,X}'$ by \cite[Corollary 3.29]{FMNT}, a contradiction.  Since $\overline{\mathfrak{T}}$ is a torsion class in $\Coh (X)$, we have $E'' \in \overline{\mathfrak{T}}$.  Therefore, all the subfactors in the relative HN filtration of $E''$ have $\mu_f=0$.  As a result,  $\wh{E''}|_s$ is a 0-dimensional sheaf for a general closed point $s \in S$ (again by \cite[Corollary 3.29]{FMNT}).  In particular, we have $\wh{E''} \in \Coh^{\leq 1}(Y)$.  By Lemma \ref{lemma75}, we can then fit $\wh{E''}$ in a short exact sequence in $\Coh (Y)$
\[
0 \to (\wh{E''})_0 \to \wh{E''} \to (\wh{E''})_1 \to 0
\]
where $(\wh{E''})_0 \in \{\Coh^{\leq 0}(Y_s)\}^\uparrow$ and $(\wh{E''})_1 \in \mathcal B_Y$.  Note that, all the terms in the above short exact sequence are $\Phi$-WIT$_0$, and so we obtain $E'' \in \langle \mathcal B_X, \Phi (\{\Coh^{\leq 0}(Y_s)\}^\uparrow)\rangle \subseteq \mathfrak T_X$.  Thus $E \in \mathfrak T_X$.
\end{proof}

\subsection{t-structures on elliptic threefolds}\label{sec-3foldequiv}

In this section, we prove an analogue of Theorem \ref{theorem1} on dual elliptic threefolds $\pi : X \to S$ and $\wh{\pi} : Y\to S$.  The strategy is the same as that on elliptic surfaces - we analyse the images of various categories of coherent sheaves  under the functor $\Lambda (-) := (\Psi (-))^\vee$.

Note that, since we defined a fiber sheaf on $X$ to be a sheaf supported on a finite number of fibers of $\pi$, now that $\pi : X \to S$ is an elliptic threefold, the category $\mathcal B_X$ coincides with $\Coh (\pi)_{\leq 1}$, and is strictly larger than the category of fiber sheaves on $X$, which is precisely $\Coh (\pi)_{\leq 0}$.

\begin{lemma}\label{lemma33}
Suppose $X$ is an elliptic threefold.  Let $E$ be any fiber sheaf on $X$.  Then:
\begin{itemize}
\item[(i)] If $E$ is 0-dimensional, then $\Lambda (E)\cong \EExt^2 (\wh{E},\OO_Y)[-2]$.
\item[(ii)] If $E$ is 1-dimensional and $\Psi$-WIT$_0$, then $\Lambda (E) \cong (\wh{E})^\vee \cong \EExt^2 (\wh{E},\OO_Y)[-2]$.
\item[(iii)] If $E$ is 1-dimensional and $\Psi$-WIT$_1$, then $\Lambda (E) \cong (\wh{E})^\vee[1]$ lies in $D^{[1,2]}_{\Coh (Y)}$, with $\Lambda^1 E \cong \EExt^2 (\wh{E},\OO_Y)$ being a pure 1-dimensional fiber sheaf (if nonzero), and $\Lambda^2 E \cong \EExt^3 (\wh{E},\OO_Y)$ being a 0-dimensional sheaf.
\end{itemize}
Overall, if $E$ is an arbitrary fiber sheaf on $X$, then $\Lambda (E)$ only has cohomology at degrees 1 and 2, and $\Lambda^2 E \in \mathcal B_{Y,\ast}$.
\end{lemma}

\begin{proof}
The proofs of statements (i), (ii) and (iii) are all straightforward, and we omit them here.  Given any fiber sheaf $E$ on $X$, i.e.\ $E \in \Coh (\pi)_0$, we can find a short exact sequence in $\Coh (X)$
\[
0 \to E^0 \to E \to E^1 \to 0
\]
where $E^0 \in \Coh (\pi)_0 \cap W_{0,X}$ and $E^1 \in \Coh (\pi)_0 \cap W_{1,X}$.  We can then find another short exact sequence in $\Coh (X)$
\[
0 \to E^{0,0} \to E^0 \to E^{0,1} \to 0
\]
where $E^{0,0} \in \Coh^{\leq 0}(X)$ and $E^{0,1} \in \Coh (\pi)_0 \cap W_{0,X} \cap \Coh^{=1}(X)$.  As a result, we obtain a filtration in $\Coh (X)$
\[
 E^{0,0} \subseteq E^0 \subseteq E
\]
where $E^{0,0} \in \Coh^{\leq 0}(X)$, $E^0/E^{0,0} \in \Coh^{=1}(X) \cap W_{0,X}$ and $E/E^0 \in \Coh^{=1}(X) \cap W_{1,X}$.  The final claim of the lemma then follows from statements (i) through (iii).
\end{proof}

\begin{lemma}\label{lemma34}
Suppose $X$ is an elliptic threefold.  Let $E \in \Coh^{\leq 1}(X)$ be such that $E$ has no fiber subsheaves.  Then:
 \begin{itemize}
 \item[(i)] $E \in \Coh(\pi)_1$.
 \item[(ii)] $E \in W_{0,X}$, and $\wh{E}$ is a 2-dimensional reflexive sheaf.
 \item[(iii)] $\Lambda (E) \cong (\wh{E})^\vee \cong   \EExt^1 (\wh{E},\OO_Y)[-1] \in \left(\Coh^{=2}(Y) \cap \Psi (\{\Coh^{\leq 0}(X_s)\}^\uparrow)\right)[-1]$.
 \end{itemize}
\end{lemma}

\begin{proof}
Part (i) follows from Lemma \ref{lemma75}.  Part (ii) is just Corollary \ref{coro3}, and part (iii) follows easily from part (ii) and \cite[Proposition 1.1.10]{HL}.
\end{proof}

\begin{lemma}\label{lemma35}
Suppose $X$ is an elliptic threefold.  Let $E \in \Coh^{=2}(X) \cap \Coh (\pi)_1$.  Then, with $E_0, E_1$ as in \eqref{eq9}, we have:
\begin{itemize}
\item[(i)] If $E_0 \neq 0$, then $\wh{E_0}$ is pure 2-dimensional and reflexive, and $\Lambda (E_0) \cong \EExt^1 (\wh{E_0},\OO_Y)[-1]$;
\item[(ii)] $\wh{E_1} \in \Coh^{\leq 2}(Y)$, and $\Lambda (E_1) \in D^{[0,2]}_{\Coh (Y)}$  with $H^0(\Lambda (E_1)) \cong \EExt^1 (\wh{E_1},\OO_Y)$ (which is pure of dimension 2 if nonzero) and $H^2 (\Lambda (E_1)) \cong \EExt^3 (\wh{E_1},\OO_Y)$.
\end{itemize}
Overall, $\Lambda (E) \in D^{[0,2]}_{\Coh (Y)}$ with $\Lambda^0 E \cong \EExt^1 (\wh{E_1},\OO_Y)$, $\Lambda^1 E \in \Coh^{\leq 2}(Y)$ and  $\Lambda^2 E \in \Coh^{\leq 0}(Y)$.
\end{lemma}

\begin{proof}
If $E_0 \neq 0$, then $E_0$ is also pure of dimension 2.  Then by Lemma \ref{lemma26}(i), $\wh{E_0}$ is pure 2-dimensional.  That $E_0$ is pure 2-dimensional also implies the vanishing $\Hom (\Phi (\Coh^{\leq 0}(Y)),E_0)=0$.  Hence $\wh{E_0}$ is reflexive by  Lemma \ref{lemma70}, and part (i) holds.

For part (ii), note that Lemma \ref{lemma66} gives
\[
\dimension (\hat{\pi}(\supp( \wh{E_1}))) = \dimension (\pi (\supp ( E_1))) = 1,
\]
 and so $\wh{E_1} \in \Coh^{\leq 2}(Y)$ and $\Lambda (E_1) \in D^{[0,2]}_{\Coh (Y)}$.  If $\wh{E_1}$ is 2-dimensional, then since it is of codimension 1, the sheaf $\EExt^1 (\wh{E_1},\OO_Y)$ is nonzero and is also pure 2-dimensional.  The rest of the lemma is clear.
\end{proof}

\begin{lemma}\label{lemma36}
Suppose $X$ is an elliptic threefold.  Let $E \in \Coh^{=2}(X) \cap \Coh (\pi)_2$, and suppose $E$ has no subsheaves lying in $\Coh (\pi)_1$.  Let $E_0, E_1$ be as in \eqref{eq9}.  Then $\wh{E_0} \in  \Coh^{=3}(Y)$, and
\[
\Lambda E \in \langle \Coh^{\geq 2}(Y), \Coh^{\leq 2}(Y)[-1], \mathcal B_{Y,\ast}[-2] \rangle.
\]
\end{lemma}

\begin{proof}
Since $\dimension E = \dimension (\pi (\supp ( E))) = \dimension S$, we know $E|_s$ is 0-dimensional for a general closed point $s \in S$.  As a result,  $E \in W_{0,X}'$ and if we let $E_0, E_1$ be as in \eqref{eq9}, then Lemma \ref{lemma29} says that $E_1 \in \mathcal B_X$.

That $E$ has no subsheaves in $\Coh (\pi)_1$ implies $E_0$ also has no subsheaves in $\Coh (\pi)_1$.  Therefore, by  Lemma \ref{lemma26}(ii), we obtain that $\wh{E_0}$ is  pure of dimension 3.  On the other hand, since $E_0$ is pure 2-dimensional,  Lemma \ref{lemma70} gives us $\EExt^2 (\wh{E_0},\OO_Y)=0$.  Thus $\Lambda E_0 \cong (\wh{E_0})^\vee \in D^{[0,1]}_{\Coh (Y)}$ with $\Lambda^0 (E_0) \in \Coh^{=3}(Y)$ and $\Lambda^1 (E_0) \in \Coh^{\leq 1}(Y)$.

Now, since $E_1 \in \mathcal B_X = \Coh (\pi)_{\leq 1}$, we can fit $E_1$ in a short exact sequence in $\Coh (X)$
\begin{equation}\label{eq59}
0 \to T_1 \to E_1 \to E_1/T_1 \to 0
\end{equation}
where $T_1 \in \Coh^{\leq 1}(X)$ and $E_1/T_1 \in \Coh^{=2}(X) \cap \Coh (\pi)_1$.  We can further fit $T_1$ in a short exact sequence in $\Coh (X)$
\begin{equation}\label{eq60}
  0 \to T_{1,f} \to T_1 \to T_1/T_{1,f} \to 0
\end{equation}
where $T_{1,f} \in \Coh (\pi)_0$, and $T_1/T_{1,f} \in \Coh^{=1}(X) \cap (\Coh (\pi)_0)^\circ$.  Now, by Lemmas \ref{lemma33}, \ref{lemma34} and \ref{lemma35}, as well as the filtrations \eqref{eq59} and \eqref{eq60}, we see that
\[
  \Lambda E_1 \in \langle \Coh^{=2}(Y), \Coh^{\leq 2}(Y)[-1], \mathcal B_{Y,\ast}[-2] \rangle.
\]
This, together with the previous paragraph, gives us
\[
  \Lambda E \in \langle \Coh^{\geq 2}(Y), \Coh^{\leq 2}(Y)[-1], \mathcal B_{Y,\ast}[-2] \rangle
\]
as claimed.
\end{proof}

\begin{lemma}\label{lemma37}
Suppose $X$ is an elliptic threefold.  Let $E \in W_{0,X}' \cap \Coh^{=3}(X)$, and let $E_0, E_1$ be as in \eqref{eq9}.  Then:
 \begin{itemize}
 \item[(i)] $E_1 \in \mathcal B_X$;
 \item[(ii)] $E_0\neq 0$  and $\wh{E_0} \in \Coh^{=3}(Y)$;
  \item[(iii)] $\Lambda (E_0) \in D^{[0,2]}_{\Coh (Y)}$ with $\Lambda^0 (E_0) \cong \EExt^0 (\wh{E_0},\OO_Y)$, $\Lambda^1 (E_0) \in \Coh^{\leq 1}(Y)$ and $\Lambda^2 (E_0) \in \Coh^{\leq 0}(Y)$.
  \item[(iv)] $\Lambda^0 (E)$ is nonzero, supported in dimension 3 and lies in $\Coh^{\geq 2}(Y)$.
  \end{itemize}
  Overall, we have $\Lambda E \in D^{[0,2]}_{\Coh (Y)}$ with $\Lambda^1 \in \Coh^{\leq 2}(Y)$ and $\Lambda^2 E \in \mathcal B_{Y,\ast}$.
\end{lemma}

\begin{proof}
Part (i) follows from Lemma \ref{lemma29}.  If $E_0 =0$, then $E = E_1$ would be torsion by part (i), a contradiction; thus $E_0 \neq 0$.  That $\wh{E_0}$ is pure 3-dimensional follows from Lemma \ref{lemma26}(i).  Thus part (ii) holds, and part (iii) follows easily.

Now, by part (i), we know $E_1 \in \mathcal B_X$.  From the second half of the proof of Lemma \ref{lemma36}, we also know that
\[
  \Lambda E_1 \in \langle \Coh^{=2}(Y), \Coh^{\leq 2}(Y)[-1], \mathcal B_{Y,\ast}[-2] \rangle,
\]
and so  $\Lambda^0 (E_1)$ is a pure 2-dimensional sheaf if nonzero.  From \eqref{eq9}, we have
the exact triangle in $D^b(Y)$
\[
 \Lambda (E_1) \to \Lambda (E) \to \Lambda (E_0) \to \Lambda (E_1)[1]
\]
and the associated long exact sequence of cohomology:
\[
 0 \to \Lambda^0 (E_1) \to \Lambda^0(E) \overset{\al}{\to} \Lambda^0 (E_0) \to \Lambda^1 (E_1) \to \cdots.
\]

From parts (ii) and (iii), we know that $\Lambda^0 (E_0)\neq 0$ is nonzero and pure 3-dimensional.  Since $\Lambda^1 (E_1) \in \Coh^{\leq 2}(Y)$ from above, we see that $\al$ is nonzero.  Thus $\image (\al)$ is nonzero and is pure 3-dimensional.  Now, we also know  $\Lambda^0 (E_1) \in \Coh^{=2}(Y)$ from above, and so   $\Lambda^0(E)$ must be nonzero, supported in dimension 3, and lies in $\Coh^{\geq 2}(Y)$.
\end{proof}

\begin{lemma}\label{lemma42}
Suppose $X$ is an elliptic threefold, and $E \in \Phi ( \{\Coh^{\leq 0}(Y_s)\}^\uparrow)$.  Then $\Lambda (E) \in D^{[0,2]}_{\Coh (Y)}$ where $\Lambda^0 (E) \cong \EExt^1 (T_2,\OO_Y)$ is a pure 2-dimensional sheaf if nonzero (here, $T_2$ denotes the pure 2-dimensional component of $\wh{E}$), $\Lambda^1 (E) \in \Coh^{\leq 1}(Y)$ and $\Lambda^2 (E) \in \Coh^{\leq 0}(Y)$.
\end{lemma}

\begin{proof}
By Lemma \ref{lemma20}, the sheaf $E$ is $\Psi$-WIT$_1$.  Thus $\Lambda (E) \cong (\wh{E}[-1])^\vee \cong (\wh{E})^\vee [1]$.  Since $\wh{E} \in \{ \Coh^{\leq 0}(Y_s)\}^\uparrow \subseteq \Coh^{\leq 2}(Y)$, we know $(\wh{E})^\vee \in D^{[1,3]}_{\Coh(Y)}$, and so $\Lambda (E) \in D^{[0,2]}_{\Coh (Y)}$.

Let $T_1$ be the maximal 1-dimensional subsheaf of $\wh{E}$, and let $T_2 := \wh{E}/T_1$.  From the short exact sequence $0 \to T_1 \to \wh{E} \to T_2 \to 0$ in $\Coh (Y)$, we obtain an exact triangle in $D(Y)$
 \[
  T_1 \to \wh{E} \to T_2 \to T_1[1].
 \]
 Dualising this exact triangle and taking the long exact sequence of cohomology, we obtain  $\Lambda^0 (E) \cong H^1 ((\wh{E})^\vee)\cong H^1 (T_2^\vee) \cong \EExt^1 (T_2,\OO_Y)$.   The rest of the lemma follows easily from the long exact sequence.
\end{proof}

\begin{lemma}\label{lemma39}
Suppose $X$ is an elliptic threefold.  Let $E \in \Coh^{=3}(X) \cap  W_{1,X}'$ and suppose $\Hom ( \Phi (\{ \Coh^{\leq 0}(Y_s)\}^\uparrow),E)=0$.  Then $E \in W_{1,X}$, and  $\Lambda (E[1]) \in D^{[0,2]}_{\Coh (Y)}$ with $\Lambda^0 (E[1]) \cong \EExt^0 (E',\OO_Y)$, where $E'$ is the torsion-free part of $\wh{E}$ and is $\Phi$-WIT$_0$, and $\Lambda^1 (E[1]) \in \Coh^{\leq 2}(Y)$ while $\Lambda^2 (E[1]) \in \mathcal B_{Y,\ast}$.
\end{lemma}

\begin{proof}
With $E$ as given, let $E_0, E_1$ be as in \eqref{eq9}.  Then $E_0 \in \mathcal B_X$ by Lemma \ref{lemma29}.  Since $E$ is pure 3-dimensional, we have $E_0=0$ and so $E \in W_{1,X}$.  Hence $\Lambda (E[1]) \cong (\wh{E})^\vee$.

That $E$ is torsion-free implies
\begin{equation}\label{eq22}
\Hom (\mathcal B_Y \cap W_{0,Y},\wh{E})=0.
\end{equation}
Now, let $T$ be the maximal torsion subsheaf of $\wh{E}$.  Note that $T$ could well be the same as $\wh{E_0}$.  Also, let $T_1$ be the maximal 1-dimensional subsheaf of $T$ and  let $T_2 := T/T_1$, which is pure 2-dimensional if nonzero.  By the vanishing \eqref{eq22}, $T_1$ is pure 1-dimensional, and must be $\Phi$-WIT$_1$ (by Remark \ref{remark21} and Lemma \ref{lemma75}).  We can consider the images of $T_1, T_2$ and $\wh{E}/T$ under derived dual  $(-)^\vee$ separately:
\begin{itemize}
\item $T_1^\vee \cong \EExt^2 (T_1,\OO_Y)[-2]\in \mathcal B_{Y,\ast}[-2]$.
\item $T_2^\vee \in D^{[1,2]}_{\Coh (Y)}$ where $H^1 (T_2^\vee) \cong \EExt^1 (T_2,\OO_Y)$ is pure 2-dimensional, and $H^2 (T_2^\vee) \cong \EExt^2 (T_2,\OO_Y)$ is 0-dimensional.
\item $(\wh{E}/T)^\vee \in D^{[0,2]}_{\Coh (Y)}$ where $H^0 ( (\wh{E}/T)^\vee) \cong \EExt^0 (\wh{E}/T,\OO_Y)$ (and $\wh{E}/T$, being the quotient of a $\Phi$-WIT$_0$ sheaf, is again $\Phi$-WIT$_0$).  Also, $H^1 ((\wh{E}/T)^\vee))\cong \EExt^1 (\wh{E}/T,\OO_Y)$ is 1-dimensional while $H^2 ((\wh{E}/T)^\vee))\cong \EExt^2 (\wh{E}/T,\OO_Y)$ is 0-dimensional.  
\end{itemize}
From above, we see that $H^0 (T^\vee)=0$.  From the short exact sequence in $\Coh (Y)$
\[
 0 \to T \to \wh{E} \to \wh{E}/T \to 0,
\]
we obtain the long exact sequence
\[
0 \to H^0 ((\wh{E}/T)^\vee) \to H^0 ((\wh{E})^\vee) \to H^0 (T^\vee) \to \cdots.
\]
Thus $H^0 ((\wh{E})^\vee) \cong H^0 ((\wh{E}/T)^\vee) \cong \EExt^0 (\wh{E}/T,\OO_Y)$ where $\wh{E}/T$ is $\Phi$-WIT$_0$ and pure 3-dimensional if nonzero.  The rest of the lemma is clear.
\end{proof}

\begin{lemma}\label{lemma48}
If $E \in \mathfrak T_X^\circ$, then $E$ satisfies the hypotheses of Lemma \ref{lemma39}.
\end{lemma}

\begin{proof}
Let $E \in \mathfrak T_X^\circ$.  Since $\mathfrak T_X$ contains all the torsion sheaves and contains $\Phi (\{\Coh^{\leq 0}(Y_s)\}^\uparrow)$, we see that $E$ is pure 3-dimensional and satisfies the vanishing $\Hom ( \Phi (\{\Coh^{\leq 0}(Y_s)\}^\uparrow),E)=0$.  Also, since $W_{0,X}' \subset \mathfrak T_X$ by definition, we have $E \in (W_{0,X}')^\circ$.  From Remark \ref{remark18}, we get $E \in W_{1,X}'$.  Thus $E$ satisfies all the hypotheses of Lemma \ref{lemma39}.
\end{proof}

\begin{remark}\label{remark19}
For ease of reference, let us list here the conclusions of some of the lemmas above.  Suppose that $\pi : X \to S$ and $\wh{\pi}: Y \to S$ are a pair of dual  elliptic threefolds.  Then
\begin{itemize}
\item Lemma \ref{lemma33}: for fiber sheaves $E$, we have $\Lambda E \in \langle \Coh (\wh{\pi})_0[-1], \mathcal B_{Y,\ast} [-2]\rangle$.
\item Lemma \ref{lemma34}: for $E \in \Coh^{\leq 1}(X)$ without fiber subsheaves, we have $\Lambda E \in \left(\Coh^{=2}(Y) \cap \Psi (\{\Coh^{\leq 0}(X_s)\}^\uparrow)\right)[-1]$.
\item Lemma \ref{lemma35}: for $E \in \Coh^{=2}(X) \cap \Coh (\pi)_1$, we have \[
    \Lambda E \in \langle \Coh^{=2}(Y), \Coh^{\leq 2}(Y)[-1], \Coh^{\leq 0}(Y)[-2]\rangle.
    \]
\item Lemma \ref{lemma36}: for $E \in \Coh^{=2}(Y) \cap \Coh (\pi)_2 \cap (\Coh (\pi)_{\leq 1})^\circ$, we have
    \[
    \Lambda E \in \langle \Coh^{\geq 2}(Y), \Coh^{\leq 2}(Y)[-1], \mathcal B_{Y,\ast}[-2]\rangle.
    \]
\item Lemma \ref{lemma37}: for $E \in W_{0,X}' \cap \Coh^{=3}(X)$, we have
\[
  \Lambda E \in \langle \Coh^{\geq 2}(Y), \Coh^{\leq 2}(Y)[-1], \mathcal B_{Y,\ast}[-2]\rangle.
\]
\item Lemma \ref{lemma42}: for $E \in \Phi ( \{\Coh^{\leq 0}(Y_s)\}^\uparrow)$, we have
    \[
      \Lambda E \in \langle  \Coh^{=2}(Y) \cap \{\Coh^{\leq 0}(Y_s)\}^\uparrow, \Coh^{\leq 1}(Y)[-1], \Coh^{\leq 0}(Y)[-2]\rangle.
    \]
\item Lemma \ref{lemma39}: for $E \in \Coh^{=3}(X) \cap W_{1,X}'$ with $\Hom ( \Phi (\{\Coh^{\leq 0}(Y_s)\}^\uparrow), E)=0$, we have
    \[
      \Lambda (E[1]) \in \langle \Coh^{=3}(Y), \Coh^{\leq 2}(Y)[-1], \mathcal B_{Y,\ast} [-2]\rangle.
    \]
\end{itemize}
\end{remark}

\begin{lemma}\label{lemma77}
Give an exact triangle $E' \to E \to E'' \to E'[1]$ in $D(X)$, if both $\Lambda E', \Lambda E''$ lie in the extension closure
\[
\langle \Coh^{\geq 2}(Y), \Coh^{\leq 2}(Y)[-1], \mathcal B_{Y,\ast}[-2]\rangle,
\]
then so does $\Lambda E$.
\end{lemma}

\begin{proof}
From the long exact sequence of cohomology for the exact triangle $E' \to E \to E'' \to E'[1]$, we have the following exact sequences:
\begin{align*}
0 \to \Lambda^0 E'' \to &\Lambda^0 E \to \Lambda^0 E', \\
 \Lambda^1 E'' \to &\Lambda^1 E \to \Lambda^1 E', \text{ and} \\
 \Lambda^2 E'' \to &\Lambda^2 E \to \Lambda^2 E' \to 0.
\end{align*}
Since $\Lambda^0 E'', \Lambda^0 E' \in \Coh^{\geq 2}(Y)$, we have $\Lambda^0 E$.  Also, that $\Lambda^1 E'', \Lambda^1 E' \in \Coh^{\leq 2}(Y)$ implies $\Lambda^1 E$ since $\Coh^{\leq 2}(Y)$ is a Serre subcategory of $\Coh (Y)$.  And finally, that $\Lambda^2 E'', \Lambda^2 E' \in \mathcal B_{Y,\ast}$ implies $\Lambda^2$ since $\mathcal B_{Y,\ast}$ is closed under quotients and extensions in $\Coh (Y)$.
\end{proof}

Combining the above lemmas together, we now have:

\begin{theorem}\label{theorem3}
Suppose $X$ is an elliptic threefold.  We have
\begin{equation}\label{eq46}
   \Lambda (\mathfrak C_X) \subseteq \langle \Coh^{\geq 2}(Y), \Coh^{\leq 2}(Y)[-1], \mathcal B_{Y,\ast}[-2]\rangle.
 \end{equation}
\end{theorem}

\begin{proof}
We begin by showing that any sheaf in $W_{0,X}'$ can be filtered by sheaves of the types from the various lemmas above.  The idea is to use the fact that $W_{0,X}'$ is a torsion class (hence closed under quotients in $\Coh (X)$), and that we have the following nested sequence of Serre subcategories in $\Coh (X)$:
\[
  \Coh (\pi)_{\leq 0} \subset \Coh^{\leq 1}(X) \subset \Coh (\pi)_{\leq 1} \subset \Coh^{\leq 2}(X).
\]
Fix any $E \in W_{0,X}'$.  Then $E$ fits in a short exact sequence in $\Coh (X)$
\[
 0 \to E^{0,0} \to E \to E^{0,1} \to 0
\]
where $E^{0,0} \in \Coh^{\leq 2}(X)$ and $E^{0,1} \in \Coh^{=3}(X) \cap W_{0,X}'$.  We can then fit $E^{0,0}$ in a short exact sequence in $\Coh (X)$
\[
0 \to E^{1,0} \to E^{0,0} \to E^{1,1} \to 0
\]
where $E^{1,0} \in \Coh^{\leq 2}(X) \cap \Coh (\pi)_{\leq 1} = \Coh(\pi)_{\leq 1}$ and $E^{1,1}$ lies in $\Coh^{\leq 2}(X) \cap (\Coh (\pi)_{\leq 1})^\circ$, which is the same as $\Coh^{=2}(X) \cap \Coh (\pi)_2 \cap (\Coh (\pi)_{\leq 1})^\circ$.  In turn, we can fit $E^{1,0}$ in a short exact sequence in $\Coh (X)$
\[
0 \to E^{2,0} \to E^{1,0} \to E^{2,1} \to 0
\]
where $E^{2,0} \in \Coh (\pi)_{\leq 1} \cap \Coh^{\leq 1}(X)=\Coh^{\leq 1}(X)$ and $E^{2,1}$ lies in $\Coh(\pi)_{\leq 1} \cap (\Coh^{\leq 1}(X))^\circ$, which is the same as $\Coh^{=2}(X) \cap \Coh (\pi)_1$.  Next, we can fit $E^{2,0}$ in a short exact sequence in $\Coh (X)$
\[
0 \to E^{3,0} \to E^{2,0} \to E^{3,1} \to 0
\]
where $E^{3,0} \in \Coh^{\leq 1}(X) \cap \Coh(\pi)_{0}$ and $E^{3,1} \in \Coh^{\leq 1}(X) \cap (\Coh (\pi)_{0})^\circ$.  Overall, we have constructed a filtration of $E$
\[
  E^{3,0} \subseteq E^{2,0} \subseteq E^{1,0} \subseteq E^{0,0} \subseteq E
\]
where the subfactors are:
\begin{itemize}
\item $E^{3,0}$, which satisfies the hypotheses of Lemma \ref{lemma33};
\item $E^{2,0}/E^{3,0} \cong E^{3,1}$, which satisfies the hypotheses of Lemma \ref{lemma34};
\item $E^{1,0}/E^{2,0} \cong E^{2,1}$, which satisfies the hypotheses of Lemma \ref{lemma35};
\item $E^{0,0}/E^{1,0} \cong E^{1,1}$, which satisfies the hypotheses of Lemma \ref{lemma36};
\item $E/E^{0,0} \cong E^{0,1}$, which satisfies the hypotheses of Lemma \ref{lemma37}.
\end{itemize}
As a result (see Remark \ref{remark19}, for instance), each of the subfactors of $E$ listed above is contained in the category
\begin{equation}\label{eq61}
\langle \Coh^{\geq 2}(Y), \Coh^{\leq 2}(Y)[-1], \mathcal B_{Y,\ast}[-2]\rangle.
\end{equation}
Therefore, $\Lambda (W_{0,X}')$ is contained in the category \eqref{eq61}.  Besides, the two categories $\Lambda (\Phi (\{ \Coh^{\leq 0}(Y_s)\}^\uparrow))$ and $\Lambda (\mathfrak T_X^\circ[1])$ are also contained in the category \eqref{eq61} by Lemmas \ref{lemma42}, \ref{lemma39} and \ref{lemma48}.  The inclusion \eqref{eq46}  thus follows from Lemma \ref{lemma77}.
\end{proof}

\begin{remark}\label{remark20}
Since $\mathcal B_{Y,\ast} \subseteq \mathfrak T_X$, Theorem \ref{theorem3} immediately gives
 \begin{equation}\label{eq52}
  \Lambda (\mathfrak C_X) \subseteq \langle \mathfrak C_Y, \mathfrak C_Y[-1],  \mathfrak C_Y [-2]\rangle.
\end{equation}
\end{remark}

Now we have the following theorem, which can be considered as the  analogue of Theorem \ref{theorem1} on elliptic threefolds:

\begin{theorem}\label{theorem2}
Suppose $X$ is a smooth elliptic threefold.  Then the heart $\Lambda (\mathfrak C_X)$ differs from  the heart $\mathfrak D_Y[-2]$ by one tilt.
\end{theorem}

\begin{proof}
Since $\mathcal B_{Y,\ast} \subseteq \Coh^{\leq 1}(Y)$, we have $\Coh^{\geq 2}(Y) \subseteq \mathcal B_{Y,\ast}^\circ$.  The theorem then follows from Theorem \ref{theorem3} and \cite[Proposition 2.3.2(b)]{BMTI}.
\end{proof}

\end{document}